\documentclass[10pt,reqno]{amsart}

\usepackage[T1]{fontenc}

\usepackage{mathtools}
\usepackage{amsmath,amsfonts,amsbsy,amsgen,amscd,mathrsfs,amssymb,amsthm}
\usepackage{enumerate}
\usepackage{bm}

\usepackage{stmaryrd}
\SetSymbolFont{stmry}{bold}{U}{stmry}{m}{n} % suppress bm warning

\usepackage{tikz}
\usetikzlibrary{decorations.pathreplacing}

\usepackage[colorlinks=true,citecolor=blue,linkcolor=blue]{hyperref}

\newtheorem{thm}{Theorem}[section]
\newtheorem{lem}[thm]{Lemma}

\newtheorem{prop}[thm]{Proposition}
\newtheorem{cor}[thm]{Corollary}

\theoremstyle{definition}

\newtheorem{defn}[thm]{Definition}

\newtheorem{model}[thm]{Model}

\theoremstyle{remark}

\newtheorem{example}[thm]{Example}

\newtheorem{rem}[thm]{Remark}
\newtheorem*{rem*}{Remark}

\numberwithin{equation}{section} 
\numberwithin{figure}{section}
\numberwithin{table}{section}

\newcommand{\ntr}{\mathop{\mathrm{tr}}}
\newcommand{\tr}{\mathop{\mathrm{Tr}}}
\newcommand{\geent}{\mathbin{\ge_{\mathrm{e}}}}
\newcommand{\leent}{\mathbin{\le_{\mathrm{e}}}}

\newcommand{\EE}{\mathbf{E}}

\begin{document}

\title[Extremal random matrices with independent entries]{
\scalebox{.9}[1.0]{Extremal random 
matrices with independent entries}
\scalebox{.9}[1.0]{and matrix superconcentration 
inequalities}}

\author{Tatiana Brailovskaya}
\address{Department of Mathematics, Duke University, 120 Science Dr,
Durham, NC 27710, USA}
\email{tatiana.brailovskaya@duke.edu}

\author{Ramon van Handel}
\address{Fine Hall 207, Princeton University, Princeton, NJ 08544, USA}
\email{rvan@math.princeton.edu}

\begin{abstract}
We prove nonasymptotic matrix concentration inequalities for the spectral 
norm of (sub)gaussian random matrices with centered independent 
entries that capture fluctuations at the Tracy-Widom scale. This 
considerably improves previous bounds in this setting due to Bandeira and 
Van Handel, and establishes the best possible tail behavior for random 
matrices with an arbitrary variance pattern. These bounds arise from an 
extremum problem for nonhomogeneous random matrices: among all variance 
patterns with a given sparsity parameter, the moments of the random matrix 
are maximized by block-diagonal matrices with i.i.d.\ entries in each 
block. As part of the proof, we obtain sharp bounds on large
moments of Gaussian Wishart matrices.
\end{abstract}

 % RMT: algebraic aspects


\maketitle

\thispagestyle{empty}
\iffalse
{\small
\setcounter{tocdepth}{2}
\tableofcontents
}
\fi

\section{Introduction}
\label{sec:intro}

Let $X$ be an $n\times m$ matrix whose entries are independent, centered 
Gaussian variables with an arbitrary variance pattern $X_{ij}\sim 
N(0,b_{ij}^2)$. This paper is concerned with bounding the spectral norm 
$\|X\|$ of such matrices.

In the special case that $b_{ij}=1$ for all $i,j$, we recover one of the 
most classical models of random matrix theory: then
$\|X\|^2=\|X^*X\|$ is the norm of a Wishart matrix with unit covariance.
In this case, classical methods of random matrix theory yield the exact 
asymptotic behavior \cite{Joh01}
\begin{equation}
\label{eq:wishart}
	\lim_{\substack{n,m\to\infty \\ m=cn}}
	\mathbf{P}\big[\|X\| > \sqrt{n}+\sqrt{m} + \tfrac{1}{2}
	(\tfrac{1}{\sqrt{n}}+\tfrac{1}{\sqrt{m}})^{1/3} s\big]
	= 1-F_1(s),
\end{equation}
where $F_1(s)$ is the distribution function of the Tracy-Widom law 
of order one. From this expression, we immediately read off that
in the above asymptotic regime $\|X\|=(1+o(1))(\sqrt{n}+\sqrt{m})$
with fluctuations of order $n^{-1/6}\vee m^{-1/6}$. It is 
important to note that the scale of the fluctuations is much smaller than 
what is predicted by general concentration of measure principles 
\cite{Led01}, which yield fluctuations of order $O(1)$ in the present 
setting. The presence of such unexpectedly small fluctuations is sometimes 
called the superconcentration phenomenon \cite{Cha14}. 

The Wishart model is amenable to explicit computations due to its simple 
structure and large degree of symmetry. However, there is little hope to 
perform explicit computations for an arbitrary variance pattern 
$(b_{ij})$, as such models can exhibit a wide variety of different 
structures and behaviors that are specific to the given pattern (for 
example, this class includes random band matrices and sparse matrices with 
an arbitrary deterministic sparsity pattern as special cases). 
Nonetheless, given their importance in many applications, it is of 
considerable interest to obtain bounds for random matrices with an 
arbitrary variance pattern.

One of the main results in this direction was obtained some years ago by 
Bandeira and the second author. To describe this result, define the 
parameters
\begin{equation}
\label{eq:sigmarect}
	\sigma_1^2 := 
	\max_{j\le m} \sum_{i\le n} b_{ij}^2,\qquad\quad
	\sigma_2^2 := 
	\max_{i\le n} \sum_{j\le m} b_{ij}^2,\qquad\quad
	\sigma_*^2 := \max_{\substack{i\le n\\j\le m}} b_{ij}^2.
\end{equation}
The following is shown in \cite[Corollary 3.11]{BvH16}.

\begin{thm}[\cite{BvH16}]
\label{thm:bvh}
Let $n\le m$, and let $X$ be the $n\times m$ random matrix whose entries 
$X_{ij}\sim N(0,b_{ij}^2)$ are independent. Then we have
\begin{equation*}
%\label{eq:bvhrect}
	\mathbf{P}\big[
	\|X\| > 
	(1+\varepsilon)(\sigma_1+\sigma_2) + \sigma_* t
	\big]
	\le
	n\,e^{-C_\varepsilon t^2}
\end{equation*}
for all $t\ge 0$ and $0<\varepsilon\le\frac{1}{2}$, where the
constant
$C_\varepsilon$ depends only on $\varepsilon$.
\end{thm}

Theorem \ref{thm:bvh} implies that
\begin{equation}
\label{eq:bvhrectmed}
	\|X\|\le 
	(1+\varepsilon)(\sigma_1+\sigma_2)
	+C_\varepsilon'\,\sigma_*\sqrt{\log n}
	\quad\text{w.h.p.}
\end{equation}
In the Wishart case $b_{ij}=1$ for all $i,j$, this
yields $\|X\|\le (1+o(1))(\sqrt{n}+\sqrt{m})$ which captures 
the exact leading order behavior in \eqref{eq:wishart}.
However, the order of the fluctuations in
Theorem \ref{thm:bvh} is \emph{much} larger than in \eqref{eq:wishart}:
the second term in \eqref{eq:bvhrectmed} diverges
as $n,m\to\infty$, while the second order term in \eqref{eq:wishart} is
$o(1)$. In particular, Theorem \ref{thm:bvh} fails to recover any form of 
superconcentration.

On the other hand, the large second term in \eqref{eq:bvhrectmed} is 
unavoidable in general for sparse matrices, as the following classical 
example illustrates.

\begin{example}
\label{ex:band}
Let $n=m$ and $b_{ij}=1_{|i-j|\le k}$, that is, $X$ is a random band 
matrix with bandwidth $2k+1$. Then we have 
$$
	\sigma_1+\sigma_2=2\sqrt{2k+1},\qquad\quad
	\|X\|\ge \max_i|X_{ii}|=
	(1+o(1))\sqrt{2\log n}
$$
with high probability as $n\to\infty$.
It follows that the second term in \eqref{eq:bvhrectmed} dominates the
behavior of $\|X\|$ in the sparse regime $k\ll\log n$.
In fact, \eqref{eq:bvhrectmed} optimally captures the phase transition 
for the leading order behavior of the norm of random band matrices (cf.\
\cite[Corollary 4.4]{BvH16} in the self-adjoint case).
\end{example}

Such examples suggest that the large scale of the fluctuations in Theorem 
\ref{thm:bvh} is a necessary feature of any bound for random matrices with 
an arbitrary variance pattern. Surprisingly, this expectation turns out to 
be incorrect. The main results of this paper will yield a considerable 
improvement on Theorem \ref{thm:bvh}, which simultaneously captures the 
fluctuations at Tracy-Widom scale of \eqref{eq:wishart} and sharpens the 
phase transition between sparse and dense matrices that is implicit in 
\eqref{eq:bvhrectmed}.

\subsection{Main results}

While the main results of this paper will be proved both for 
non-self-adjoint and for self-adjoint random matrices, we focus the 
presentation in the introduction on the non-self-adjoint case for 
concreteness. We further consider a slightly more general setting
than the above Gaussian model.

\begin{model}
\label{mod:rectintro}
$X$ is an $n\times m$ matrix with $X_{ij}=b_{ij}\xi_{ij}$,
where $b_{ij}\ge 0$ are arbitrary scalars and
$\xi_{ij}$ are independent symmetrically distributed
real random variables with $\mathbf{E}[\xi_{ij}^{2p}]\le
\mathbf{E}[g^{2p}]$ for all $i,j$
and $p\in\mathbb{N}$ (here $g\sim N(0,1)$).
\end{model}

In the setting of Model \ref{mod:rectintro}, we always define
$\sigma_1,\sigma_2,\sigma_*$ as in \eqref{eq:sigmarect}. 

\subsubsection{Small deviations}
\label{sec:introsm}

Our main result in this setting is the following.

\begin{thm}[Small deviations]
\label{thm:smrect}
Let $X$ be as in Model \ref{mod:rectintro} and
$\sigma_1\le\sigma_2$. Then
$$
	\mathbf{P}\big[\|X\| > \sigma_1 + \sigma_2 + 
	\sigma_*^{4/3}
	\sigma_1^{-1/3} t\big]
	\le \frac{n\sigma_*^2}{C\sigma_1^2}\, e^{-Ct^{3/2}}
$$
for all $0\le t\le \frac{\sigma_1^{1/3}\sigma_2}{\sigma_*^{4/3}}$,
where $C$ is a universal constant.
\end{thm}

\begin{rem}
The assumption $\sigma_1\le\sigma_2$ entails no loss of generality, as in
the opposite case $\sigma_1>\sigma_2$ we may simply apply Theorem 
\ref{thm:smrect} to the adjoint matrix $X^*$.
\end{rem}

In the Wishart case $b_{ij}=1$ for all $i,j$ with $n\le m$, Theorem 
\ref{thm:smrect} yields
$$
	\mathbf{P}\big[\|X\| > \sqrt{n} + \sqrt{m} +
	n^{-1/6}t\big]
	\le C^{-1}e^{-Ct^{3/2}}
$$
for $0\le t\le n^{1/6}\sqrt{m}$. This captures precisely the 
fluctuations and the upper tail of the Tracy-Widom asymptotics 
\eqref{eq:wishart}, up to a universal constant. 
To the best of our knowledge, the nonasymptotic bound 
is new even in this very special case (a slightly weaker bound in
this setting may be found in \cite[Theorem 2]{LR10}).

For arbitrary variance patterns with $\sigma_1\le\sigma_2$,
Theorem \ref{thm:smrect} yields
\begin{equation}
\label{eq:smrectmed}
	\|X\| \le \sigma_1+\sigma_2 +
	C'\,\frac{\sigma_*^{4/3}}{\sigma_1^{1/3}}
	\log^{2/3}\bigg(
	\frac{n\sigma_*^2}{\sigma_1^2}
	\bigg)
	\quad\text{w.h.p.}
\end{equation}
provided $\sigma_1^{1/4}\sigma_2^{3/4}\gtrsim \sigma_*\sqrt{\log n}$.
In particular, as soon as $\sigma_1\gg \sigma_*\sqrt{\log n}$,
\eqref{eq:smrectmed} improves drastically on 
\eqref{eq:bvhrectmed}. We will explain in 
section \ref{sec:introext} that this bound is essentially the best
possible even for nonhomogeneous random matrices.

\subsubsection{Large deviations}
\label{sec:introlg}

Theorem \ref{thm:smrect} controls small deviations, that is, deviations up 
to the order of the mean. In contrast, the large deviations of $\|X\|$ are 
controlled by its Gaussian tail behavior \cite[Corollary 3.2]{LT91}, and 
we cannot expect a qualitative improvement over Theorem \ref{thm:bvh} in 
this setting.

Nonetheless, the basic principle behind Theorem \ref{thm:smrect} gives 
rise to a quantitative improvement: a variant of Theorem \ref{thm:bvh} 
with optimal constants.

\begin{thm}[Large deviations]
\label{thm:lgrect}
Let $X$ be as in Model \ref{mod:rectintro} and
$n\le m$. Then
$$
	\mathbf{P}\big[\|X\| > \sigma_1 + \sigma_2 + \sigma_* (1+t)\big]
	\le 2n\,e^{-t^2/2}
$$
for all $t\ge 0$.
\end{thm}

Theorem \ref{thm:lgrect} implies that
\begin{equation}
\label{eq:lgrectmed}
	\|X\| \le \sigma_1+\sigma_2 + (1+o(1))\,\sigma_*\sqrt{2\log n}
	\quad\text{w.h.p.}
\end{equation}
as $n\to\infty$. The significance of this result is that in many examples,
the second term in \eqref{eq:lgrectmed} matches the trivial 
lower bound $\|X\|\ge \max_{i,j}|X_{ij}|$; this is the case, 
for example, for the random band matrix of Example \ref{ex:band}.
In such situations, \eqref{eq:lgrectmed} captures the exact leading order 
behavior $\|X\|=(1+o(1))\sigma_*\sqrt{2\log n}$
in the sparse regime
$\sigma_1+\sigma_2\ll\sigma_*\sqrt{\log n}$, while \eqref{eq:bvhrectmed} 
necessarily loses a universal constant.

\subsubsection{Extremum principle}
\label{sec:introext}

The proofs of Theorems \ref{thm:smrect} and \ref{thm:lgrect} are based on 
a more fundamental principle that is of independent interest.

To explain this principle, let us begin by recalling the idea behind the 
proof of Theorem \ref{thm:bvh}. Rather than bound the norm of the 
nonhomogeneous model $X$ directly, the approach of \cite{BvH16} compares 
$X$ with an i.i.d.\ matrix of smaller dimension, whose norm can be 
controlled by any classical method for homogeneous random matrices. In 
particular, the main technical device in \cite{BvH16} shows that when 
normalized so that $\sigma_*=1$, the $2p$-moment of $X$ can be bounded by 
the $2p$-moment of a matrix with i.i.d.\ entries of dimension $\lceil 
\sigma_1^2+p\rceil\times\lceil\sigma_2^2+p\rceil$. This suffices to 
capture the leading order behavior of $\|X\|$, but the dependence of the 
comparison matrix on $p$ precludes any accurate control of the 
fluctuations.

The basis for the results of this paper may be viewed as an 
\emph{optimal} comparison principle of this kind. In the following, we 
denote by $\tr M$ the (unnormalized) trace and by $\ntr M := 
\frac{1}{n}\tr M$ the normalized trace of an $n\times n$ matrix $M$.

\begin{thm}[Extremum principle]
\label{thm:rectext}
Let $X$ be as in Model \ref{mod:rectintro} with
$\sigma_*=1$. Then
$$
	\mathbf{E}[\ntr (XX^*)^{p}] \le \mathbf{E}[\ntr (YY^*)^{p}]
	\quad\text{for all }p\in\mathbb{N},
$$
where $Y$ is the
$\lceil\sigma_1^2\rceil \times \lceil\sigma_2^2\rceil$ matrix with
independent entries $Y_{ij}\sim N(0,1)$.
\end{thm}

Note that the inequality in Theorem \ref{thm:rectext} holds with equality 
when $X$ has i.i.d.\ entries. Thus, in contrast to the comparison 
principle of \cite{BvH16}, Theorem \ref{thm:rectext} may be viewed as an 
extremum principle for random matrices. In terms of normalized moments, 
this principle may be expressed as follows.

\begin{cor}
\label{cor:rexn}
Fix $d_1,d_2\in\mathbb{N}$.
Among all $X$ as in Model \ref{mod:rectintro} with $\sigma_*^2=1$,
$\sigma_1^2\le d_1$, $\sigma_2^2\le d_2$, and arbitrary dimensions
$n\times m$, the normalized moments
$\mathbf{E}[\ntr (XX^*)^{p}]$ are maximized for all $p$ by the  
matrix of dimension $d_1\times d_2$ with i.i.d.\ $N(0,1)$
entries.
\end{cor}

In terms of unnormalized moments, we obtain the
following.

\begin{cor}
\label{cor:rexu}
Fix $n,m,d_1,d_2\in\mathbb{N}$ with $\frac{m}{d_2}\ge
\frac{n}{d_1}\in\mathbb{N}$. Among all
$X$ as in Model~\ref{mod:rectintro} with $\sigma_*^2=1$,
$\sigma_1^2\le d_1$, $\sigma_2^2\le d_2$, and fixed dimensions $n\times
m$, the unnormalized moments $\mathbf{E}[\tr (XX^*)^{p}]$ are
maximized
for all $p$ by the block-diagonal matrix whose blocks have dimension
$d_1\times d_2$ and i.i.d.\ $N(0,1)$ entries
(Figure \ref{fig:blockd}).
\end{cor}
\begin{figure}
\begin{tikzpicture}
\draw[thick] (-.1,0) -- (-.2,0) -- (-.2,2) -- (-.1,2);
\draw[thick] (5.9,0) -- (6,0) -- (6,2) -- (5.9,2);
\draw[thick,fill=gray!15] (0,2) rectangle (1.25,1.5);
\draw[thick,fill=gray!15] (1.25,1.5) rectangle (2.5,1);
\draw[thick,fill=gray!15] (3.75,0.5) rectangle (5,0);
\draw[fill=black] (2.8125,0.875) circle (0.02);
\draw[fill=black] (3.125,0.75) circle (0.02);
\draw[fill=black] (3.4375,0.625) circle (0.02);

\draw[thick,decorate,decoration=brace] (-.5,0) -- (-.5,2);
\draw[thick,decorate,decoration=brace] (0,2.3) -- (5.8,2.3);

\draw (-.8,1) node {$\scriptstyle n$};
\draw (2.9,2.6) node {$\scriptstyle m$};
\draw (2.45,1.25) node[right] {$\scriptstyle d_1$};
\draw (1.875,1.45) node[above] {$\scriptstyle d_2$};

\end{tikzpicture}
\caption{An extremal block-diagonal matrix.\label{fig:blockd}}
\end{figure}
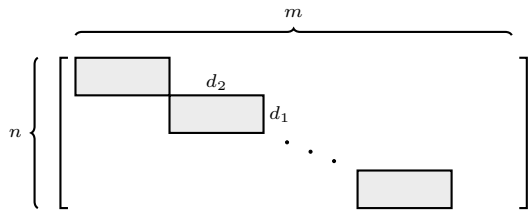

In particular, that block-diagonal matrices have the largest moments 
explains the form of Theorem \ref{thm:smrect}. Indeed, as the norm of a 
block-diagonal matrix is the maximum of the norms of its blocks, 
\eqref{eq:wishart} suggests that its distribution is approximately the 
maximum of $\frac{n}{d_1}$ independent Tracy-Widom variables. The tail 
probabilities of this distribution are precisely of the form that is 
captured by the tail bound of Theorem \ref{thm:smrect}. In particular, the 
bound \eqref{eq:smrectmed} is essentially the best possible in that it 
yields the correct behavior of block-diagonal matrices.

Theorem \ref{thm:rectext} will be proved in section \ref{sec:rect} below. 
With this result in hand, the proof of Theorems \ref{thm:smrect} and 
\ref{thm:lgrect} will reduce to the proof of analogous tail bounds for 
Wishart matrices. The latter is made possible by a beautiful method, 
pioneered by Ledoux \cite{Led04,Led07,Led09}, for deriving Tracy-Widom 
type tail bounds from sharp moment estimates.\footnote{
	Tail bounds for Wigner and Wishart matrices have also been
	proved by methods that are not based on moments, such as in 
	\cite{Aub05,LR10}. However, as Theorem \ref{thm:rectext} only
	provides a comparison principle for moments and not for tail
	probabilities, the moment approach is essential here.}
While such 
estimates were 
obtained by Ledoux for GUE \cite{Led04} and GOE \cite{Led09} matrices, 
sharp moment bounds for Wishart matrices do not appear to be known in the 
literature. Such bounds are obtained here in section \ref{sec:wishart}. 
The Wishart model is considerably more delicate than the GUE 
and GOE models due to the fact that we require bounds that hold uniformly 
for all dimensions $n,m$, not just in the asymptotic regime 
\eqref{eq:wishart} where $m=cn$ are proportional.

Let us emphasize that, despite the apparent similarity between 
Theorem~\ref{thm:rectext} and the comparison principles of \cite{BvH16}, 
the proofs of these results are completely different. The method of 
\cite{BvH16} uses a classical combinatorial expression for the moments as 
a sum over equivalence classes of even closed walks, and estimates each 
term separately. In contrast, the proof of Theorem \ref{thm:rectext} 
expresses the moments of a Gaussian random matrix as a sum over pairings 
by the Wick formula, and then uses an iterative procedure to reduce each 
term to a noncrossing pairing. This new approach turns out to provide a 
much more efficient mechanism for controlling the moments.

\subsubsection{Self-adjoint models and outline}

While we have focused the discussion in the introduction on 
non-self-adjoint models for concreteness, analogous results hold also for 
self-adjoint models. Beside that self-adjoint random matrices arise 
frequently in applications, the proofs of the self-adjoint analogues of 
our main results turn out to be simpler than those in the 
non-self-adjoint case. We will therefore develop these results
in detail before we return to the non-self-adjoint case.
\medskip
\begin{enumerate}[$\bullet$]
\itemsep\medskipamount
\item
In section \ref{sec:cplx} we consider self-adjoint random matrices with 
independent complex Gaussian entries.  The extremum principle admits a 
particularly simple proof in this setting that avoids almost all the 
complications that arise in the real case.
\item
We then consider in section \ref{sec:real} self-adjoint random matrices 
with independent real entries, that is, the self-adjoint analogue of Model 
\ref{mod:rectintro}. The proof of the the extremum principle is much more 
involved in this case, and shares the same difficulties as the proof of 
the extremum principle for non-self-adjoint models. However, for both 
complex and real self-adjoint models, we can directly invoke the moment 
estimates of Ledoux  \cite{Led04,Led07,Led09} to obtain tail bounds.
\item 
Finally, the non-self-adjoint case is developed in section 
\ref{sec:rect}, while section \ref{sec:wishart} developes the sharp moment 
estimates that are needed in this case.
\end{enumerate}

\subsection{Discussion and open questions}

\subsubsection{Intrinsic freeness}
\label{sec:free}

To date, two main approaches have been developed for obtaining sharp 
nonasymptotic bounds for the spectral norm of nonhomogeneous random 
matrices $X$ that are of a fundamentally different nature:
\medskip
\begin{enumerate}[1.]
\itemsep\medskipamount
\item The extremum principles of \cite{BvH16} and of this paper
compare the spectral statistics of $X$ with those of an i.i.d.\ random 
matrix $Y$ of smaller dimension.
\item The intrinsic freeness principle of \cite{BBV23} compares the
spectral statistics of $X$ with those of a deterministic
operator $X_{\rm free}$ that arises in free probability theory.
\end{enumerate}
\medskip
The theory of \cite{BBV23} is in fact much more general than 
the setting of this paper, in that it captures both random matrices with 
dependent entries and more general spectral statistics. However, the 
following special case may 
be directly compared to the bounds of this paper: if $X$ is an $n\times m$ 
random matrix with independent Gaussian entries $X_{ij}\sim 
N(0,b_{ij}^2)$ and $n\le m$, then \cite[Corollary 2.2]{BBV23} yields
\begin{equation}
\label{eq:freemed}
	\|X\| \le (1+\varepsilon)\|X_{\rm free}\| 
	+ C_\varepsilon\,\sigma_*(\log m)^{3/2}\quad\text{w.h.p.}
\end{equation}
for every $\varepsilon>0$ 
(this may be shown as in \cite[Lemma 3.1]{BBV23}). As 
$$
	\|X_{\rm free}\| \le \sigma_1+\sigma_2
$$
by \cite[Lemma 2.5]{BBV23}, the intrinsic freeness principle readily 
yields a
slightly weaker form of \eqref{eq:bvhrectmed}. Note that the second term 
in \eqref{eq:freemed} is even larger than that of \eqref{eq:bvhrectmed}, 
while the results of this paper yield a much smaller
second-order term.

On the other hand, there are many situations where the intrinsic freeness
principle yields strictly better results than those of this paper: 
whenever
$$
	\sigma_*(\log m)^{3/2}
	\ll\|X_{\rm free}\|\le (1-\delta)(\sigma_1+\sigma_2)
$$
for some $\delta>0$,
the bound of \eqref{eq:freemed} is strictly better to \emph{leading} order,
which renders any improvement to the second-order term negligible.

In view of the above discussion, the bounds of this paper are of 
particular interest when $\|X_{\rm free}\|=\sigma_1+\sigma_2$.  The 
following lemma explains when this is the case.

\begin{lem}
\label{lem:xfree}
$\|X_{\rm free}\|=\sigma_1+\sigma_2$ when
$\sum_i b_{ij}^2 = \sigma_1^2$ and $\sum_j b_{ij}^2=\sigma_2^2$ for all 
$i,j$.
\end{lem}

\begin{proof}
We can apply \cite[Lemma 3.2 and Remark 2.6]{BBV23} to obtain
$$
	\|X_{\rm free}\| \ge
	2\sum_i\sqrt{w_i\sum_j b_{ij}^2 v_j} +
	2\sum_j\sqrt{w_i\sum_i b_{ij}^2 v_j}
$$
for every $w\in\mathbb{R}^n_+$, $v\in\mathbb{R}^m_+$ with
$\sum_i w_i + \sum_j v_j=1$. The conclusion follows by choosing
$w_i=\frac{1}{2n}$, $v_j=\frac{1}{2m}$, and noting that 
$\sum_{i,j}b_{ij}^2 = m\sigma_1^2=n\sigma_2^2$.
\end{proof}

Variance patterns with constant row and column sums as in Lemma 
\ref{lem:xfree} arise naturally in applications, for example, in the study 
of sparse random matrices whose sparsity pattern is biregular. At the 
same time, the results of this paper are stronger than those obtained by 
the intrinsic freeness principle for very sparse matrices (so that they 
capture phase transitions as in Example \ref{ex:band}), and provide the 
strongest easily computable norm bounds for arbitrary variance patterns.

One may wonder whether it might in fact be possible to achieve the best of 
both worlds: could the tail bound of Theorem \ref{thm:smrect} with 
fluctuations at Tracy-Widom scale remain valid if the leading term 
$\sigma_1+\sigma_2$ is replaced by $\|X_{\rm free}\|$? The following 
example shows that such a bound cannot exist.

\begin{example}
\label{ex:bbp}
Let $n=m$ and $b_{ij}^2=n^{-1}(1 + \delta^2 1_{i=1})$, known as the 
spiked 
covariance model (originally due to Baik, Ben Arous, and P\'ech\'e in 
the complex case).
Then the leading order behavior of $\|X\|$ exhibits the following
phase transition:
$$
	\|X\| = (1+o(1))\|X_{\rm free}\| =
	(1+o(1))\begin{cases}
	2 & \text{if }\delta\le 1,\\
	\delta+\frac{1}{\delta} & \text{if }\delta>1
	\end{cases}
	\qquad\text{w.h.p.}
$$
as $n\to\infty$ (cf.\ \cite{BBV23,BCSV23} and \cite[\S 2.1]{Pau07}).
Moreover, it is shown in \cite[Theorem 3]{Pau07} that when $\delta>1$, the 
fluctuations of $\|X\|$ are of order $\sigma_*\sim n^{-1/2}$. This shows
that we cannot replace $\sigma_1+\sigma_2$ by $\|X_{\rm free}\|$ in
Theorem \ref{thm:smrect}, as that would imply a much smaller bound on the 
fluctuations of order $\sigma_*^{4/3}$.
\end{example}

Example \ref{ex:bbp} illustrates in a particularly clear manner that the 
extremum and intrinsic freeness principles capture fundamentally different 
mathematical phenomena. This may appear somewhat surprising, given that 
both the proofs in this paper and those of \cite{BBV23} are based on the 
control of crossings in the Wick formula (in the latter case, this idea is 
due to \cite{Tro18}). However, these results exploit crossings in very 
different ways: here we show that the contribution of each crossing to the 
nonhomogeneous model is bounded by that in the i.i.d.\ model, while the 
intrinsic freeness principle aims to show that crossings are negligible 
altogether.

\subsubsection{Superconcentration}

The concentration of measure phenomenon \cite{Led01} is a powerful tool 
for bounding the fluctuations of random structures. While general 
concentration inequalities often yield optimal bounds, there are also 
situations where the scale of the flucutations is much smaller than is 
predicted by general principles. This phenomenon was called 
superconcentration in \cite{Cha14}.

When applied to the random matrix models of this paper, classical 
concentration inequalities ensure that the scale of the fluctuations of 
$\|X\|$ is at most of order $\sigma_*$, cf.\ \cite[Corollary 4.14]{BBV23}. 
In contrast, the fluctuation term in Theorem~\ref{thm:smrect}, of order 
$\sigma_*^{4/3}\sigma_1^{-1/3}$, is often \emph{much} smaller than 
$\sigma_*$. We may therefore think of Theorem~\ref{thm:smrect} as a 
``matrix superconcentration inequality''. Indeed, this bound accurately 
captures the well known superconcentration property of the upper tail of 
the norm of Wishart matrices, and extends it to a much larger class of 
models.\footnote{%
After the completion of this paper, we learned that a weaker
superconcentration property was previously obtained by entirely different 
methods in \cite[Theorem 3.4]{EK11} for self-adjoint random matrices 
satisfying the analogue of Lemma \ref{lem:xfree}. Theorem \ref{thm:symmsm} 
below yields a much more precise tail bound in the self-adjoint setting 
that is essentially optimal by the extremum principle.}

This interpretation must be treated with care, however, as the 
second-order term in Theorem \ref{thm:smrect} can only provide information 
on the fluctuations of $\|X\|$ when the first-order term is captured 
correctly. In particular, Theorem \ref{thm:smrect} can only establish 
genuine superconcentration for variance patterns that satisfy conditions 
as in Lemma \ref{lem:xfree}. This is necessarily the case: Example 
\ref{ex:bbp} illustrates that general variance patterns may not 
exhibit any superconcentration at all.

More generally, it is not expected that the specific setting of Lemma 
\ref{lem:xfree} is necessary for superconcentration. General principles 
explained in \cite{Cha14} suggest that the spectral norm should exhibit 
superconcentration provided there are many singular values near the 
maximal one (this fails in Example \ref{ex:bbp}, where the largest 
singular value is isolated from the bulk). Furthermore, even in the 
setting of Lemma~\ref{lem:xfree}, our results yield an upper bound 
that is sharp for block-diagonal matrices, but other such models may 
exhibit even smaller fluctuations (e.g., \cite{Sod10}). A precise 
understanding of when and how much superconcentration arises for 
nonhomogeneous random matrices remains out of reach of any known method.

\subsubsection{Universality}

The setting of Model \ref{mod:rectintro} requires the entries of $X$ to 
be symmetrically distributed and have all their moments dominated by those 
of the Gaussian distribution. However, the classical Tracy-Widom 
asymptotics \eqref{eq:wishart} remain valid in a much more general setting 
\cite{PY14}: it suffices that the entries of $X$ have zero mean, unit 
variance, and soft control of the higher moments. 

It is of considerable interest to understand whether nonasymptotic bounds 
can also be achieved in this much more general setting. An extension of 
Theorem \ref{thm:bvh} along these lines was proved in \cite[\S 
4.3]{LVY18}: it was shown there that for any random matrix $X$ whose 
entries are independent and have zero mean, the statement of Theorem 
\ref{thm:bvh} remains valid if we replace the parameters 
\eqref{eq:sigmarect} by
$$
	\sigma_1^2 \leftarrow
	\max_j \sum_i \mathrm{Var}(X_{ij}),\qquad
	\sigma_2^2 \leftarrow 
	\max_i \sum_j \mathrm{Var}(X_{ij}),\qquad
	\sigma_*^2 \leftarrow \max_{i,j}
	\|X_{ij}\|_\infty^2.
$$
However, just as Theorem \ref{thm:bvh}, this result fails to capture 
fluctuations at the Tracy-Widom scale. It is natural to 
conjecture the validity of an analogous extension of Theorem 
\ref{thm:smrect} and of the other main results of this paper. Such results 
would be of particular interest in applications, e.g., to random graphs 
\cite{BBK20}.

Unfortunately, it is not clear whether the methods of this paper can be 
adapted to achieve such results (nor do they follow from the very general 
universality principles in \cite{BV23}, whose error terms are far larger 
than the fluctuations in Theorem \ref{thm:smrect}). While much less 
precise, the method used in \cite{BvH16,LVY18} does not use any special 
properties of the Gaussian distribution and is therefore readily adapted 
to more general models. In contrast, the proof of the sharp extremum 
principle of Theorem \ref{thm:rectext} is based on the Wick formula, and 
is therefore inherently Gaussian in nature. For this reason, an analogous
extension of our main results remains an open problem.

\subsection{Notation}

The following notation will be used throughout this paper. 

We denote by $N(0,\sigma^2)$ or by $N_{\mathbb{R}}(0,\sigma^2)$ the 
distribution of a real Gaussian random variable with mean $0$ and 
variance $\sigma^2$. The distribution of a complex random variable, whose 
real and imaginary parts are independent with distribution 
$N_\mathbb{R}(0,\frac{\sigma^2}{2})$ (i.e., a complex Gaussian variable), 
is denoted as $N_{\mathbb{C}}(0,\sigma^2)$.

For a matrix $M$, its adjoint is denoted as $M^*$; in particular, $a^*$ 
denotes the complex conjugate of $a\in\mathbb{C}$. We always denote by 
$\|M\|$ the spectral norm (i.e., the largest singular value) of $M$. 
Recall that for a square matrix $M$, we denote by $\tr M$ and $\ntr M$ the 
unnormalized and normalized trace, respectively.

We write $[n]:=\{1,\ldots,n\}$ for $n\in\mathbb{N}$. We denote by
$\mathrm{P}_2([n])$ the set of all pairings of $[n]$ (that is, partitions 
each of whose elements has size two). Recall that given any
$\pi\in\mathrm{P}_2([n])$, two pairs $\{i,k\},\{j,l\}\in\pi$ such that
$i<j<k<l$ are said to form a crossing. A pairing that contains no 
crossing is said to be noncrossing, and we denote by $\mathrm{NC}_2([n])$ 
the set of all noncrossing pairings of $[n]$.

Finally, we will write
$x\lesssim y$ to indicate that $x\le Cy$ for a universal constant $C$.

\section{The Hermitian case}
\label{sec:cplx}

The aim of this section is to investigate Hermitian random matrices with 
independent complex Gaussian entries. That is, we consider the 
following model.

\begin{model}
\label{mod:herm}
$X$ is an $n\times n$ Hermitian matrix whose entries $X_{ij}=(X_{ji})^*$ 
are independent for $i\ge j$ with $X_{ij}\sim N_{\mathbb{C}}(0,b_{ij}^2)$ 
for $i>j$ and $X_{ii}\sim N_{\mathbb{R}}(0,b_{ii}^2)$, where 
$b_{ij}=b_{ji}\ge 0$ are arbitrary nonnegative scalars.
\end{model}

This setting admits particularly simple proofs, which will guide the more 
involved arguments required for real random matrices in the following 
sections.

Define the parameters
\begin{equation}
\label{eq:sigmaparm}
	\sigma^2 := \max_{i\le n}\sum_{j\le n} b_{ij}^2,\qquad\qquad
	\sigma_*^2 := \max_{i,j\le n} b_{ij}^2.
\end{equation}
Then we have the following extremum principle and tail bounds.

\begin{thm}[Extremum principle]
\label{thm:hermmom}
Define $X$ as in Model \ref{mod:herm}, and assume that
$\sigma_*^2=1$ and that $\sigma^2\le d\in\mathbb{N}$. Then we have
$$
	\mathbf{E}[\ntr X^{2p}] \le \mathbf{E}[\ntr Y^{2p}]
$$
for all $p\in\mathbb{N}$, where $Y$ is the $d\times d$ Hermitian matrix
whose entries $Y_{ij}=(Y_{ji})^*$ are independent for $i\ge j$ with
$Y_{ij}\sim N_{\mathbb{C}}(0,1)$ for $i>j$
and $Y_{ii}\sim N_{\mathbb{R}}(0,1)$.
\end{thm}

\begin{thm}[Small deviations]
\label{thm:hermsm}
For $X$ as in Model \ref{mod:herm}, we have
$$
	\mathbf{P}\big[\|X\|>
	2\sigma + 4\sigma_*^{4/3}\sigma^{-1/3} t
	\big] 
	\le
	\frac{en\sigma_*^2}{\sigma^2}
	e^{-t^{3/2}}
$$
for every $0\le t\le \frac{\sigma^{4/3}}{\sigma_*^{4/3}}$.
\end{thm}

\begin{thm}[Large deviations]
\label{thm:hermlg}
For $X$ as in Model \ref{mod:herm}, we have
$$
	\mathbf{P}\big[
	\|X\|> 2\sigma + \sigma_*(1+t)\big] 
	\le
	2n\, e^{-t^2/2}
$$
for every $t\ge 0$.
\end{thm}

The remainder of this section is devoted to the proofs of these results.

\subsection{Extremum principle}

In this subsection, we will use the following notation.
Let $(g_{ij})_{i\ge j}$ be i.i.d.\ $N_{\mathbb{C}}(0,1)$ random variables, 
and define
$$
	U_{ij} := b_{ij} e_i e_j^*,\qquad\qquad
	U_{ii} := \tfrac{1}{\sqrt{2}}b_{ii} e_i e_i^*
$$
for $i>j$. Then we can represent the random matrix $X$ of
Model \ref{mod:herm} as
$$
	X = \sum_{i\ge j} ( g_{ij} U_{ij} + g_{ij}^* U_{ij}^* ).
$$
We can compute the moments of $X$ as follows.

\begin{lem}[Wick formula]
\label{lem:cplxwick}
For any $p\in\mathbb{N}$, we have
$$
	\EE[\ntr X^{2p}] =
	\sum_{\pi\in \mathrm{P}_2([2p])}
	\sum_{(\boldsymbol{i},\boldsymbol{j},\boldsymbol{\varepsilon})\sim\pi}
	\ntr U_{i_1 j_1}^{\varepsilon_1}\cdots 
	U_{i_{2p}j_{2p}}^{\varepsilon_{2p}}.
$$
Here $\boldsymbol{i}=(i_1,\ldots,i_{2p})\in [n]^{2p}$, 
$\boldsymbol{j}=(j_1,\ldots,j_{2p})\in [n]^{2p}$, 
$\boldsymbol{\varepsilon}=(\varepsilon_1,\ldots,\varepsilon_{2p})\in
\{1,*\}^{2p}$, and 
$(\boldsymbol{i},\boldsymbol{j},\boldsymbol{\varepsilon})\sim\pi$ denotes
$i_k\ge j_k$ and
$i_k=i_l,j_k=j_l,\varepsilon_k\ne\varepsilon_l$ for all 
$\{k,l\}\in\pi$.
\end{lem}

\begin{proof}
Clearly
$$
	\EE[\ntr X^{2p}] =
	\sum_{i_1\ge j_1,\ldots,i_{2p}\ge j_{2p}}
	\sum_{\varepsilon_1,\ldots,\varepsilon_{2p}\in\{1,*\}}
	\mathbf{E}[g_{i_1j_1}^{\varepsilon_1}\cdots
	g_{i_{2p}j_{2p}}^{\varepsilon_{2p}}]
	\ntr U_{i_1 j_1}^{\varepsilon_1}\cdots
        U_{i_{2p}j_{2p}}^{\varepsilon_{2p}}.
$$
The conclusion follows as
$$
	\mathbf{E}[g_{i_1j_1}^{\varepsilon_1}\cdots
        g_{i_{2p}j_{2p}}^{\varepsilon_{2p}}] =
	\sum_{\pi\in \mathrm{P}_2([2p])}
	\prod_{\{k,l\}\in\pi}
	1_{i_k=i_l,j_k=j_l,\varepsilon_k\ne\varepsilon_l}
$$
by the classical Wick formula \cite[Theorem 22.3 and Remark 22.5]{NS06}.
\end{proof}

Free probability theory suggests \cite[p.\ 367]{NS06} that the expression 
in Lemma \ref{lem:cplxwick} should be dominated by noncrossing pairings. 
The idea behind the proof of Theorem \ref{thm:hermmom} is to consider 
separately the effect of noncrossing and crossing pairs on the terms in 
the sum in Lemma \ref{lem:cplxwick}. We first consider noncrossing 
pairings.

\begin{lem}
\label{lem:cplxnc}
For any $p\in\mathbb{N}$ and noncrossing 
pairing $\pi\in\mathrm{NC}_2([2p])$, we have
$$
	\sum_{(\boldsymbol{i},\boldsymbol{j},\boldsymbol{\varepsilon})\sim\pi}
	\ntr U_{i_1 j_1}^{\varepsilon_1}\cdots 
	U_{i_{2p}j_{2p}}^{\varepsilon_{2p}}
	\le
	\sigma^{2p}.
$$
Moreover, equality holds when $b_{ij}=1$ for all $i,j$.
\end{lem}

\begin{proof}
Recall the elementary fact that any noncrossing pairing $\pi$ 
must contain a consecutive pair $\{k,k+1\}\in\pi$. (If not,
choose $\{k,k+l\}\in \pi$ with minimal $l\ge 2$. Then $\{k+1,r\}\in\pi$
must satisfy $|k+1-r|\ge l$. In particular, $\{k,k+l\}$ and $\{k+1,r\}$ 
form a crossing, contradicting the assumption.)

By cyclic permutation of the trace, we may assume without loss of 
generality that $\{2p-1,2p\}\in \pi$. Then we can compute
\begin{multline*}
	\sum_{(\boldsymbol{i},\boldsymbol{j},\boldsymbol{\varepsilon})\sim\pi}
	\ntr U_{i_1 j_1}^{\varepsilon_1}\cdots 
	U_{i_{2p}j_{2p}}^{\varepsilon_{2p}}
	= \\
	\sum_{(\boldsymbol{i},\boldsymbol{j},\boldsymbol{\varepsilon})\sim\pi
	\backslash\{\{2p-1,2p\}\}}
	\ntr 
	U_{i_1 j_1}^{\varepsilon_1}\cdots 
	U_{i_{2p-2}j_{2p-2}}^{\varepsilon_{2p-2}}
	\sum_{i\ge j}(U_{ij}U_{ij}^*+U_{ij}^*U_{ij}).
\end{multline*}
Now note that
$$
	\sum_{i\ge j}(U_{ij}U_{ij}^*+U_{ij}^*U_{ij}) =
	\sum_{i\le n} \bigg(\sum_{j\le n} b_{ij}^2\bigg) e_ie_i^*.
$$
As all $U_{ij}$ have nonnegative entries, we have
$\ntr U_{i_1 j_1}^{\varepsilon_1}\cdots
U_{i_{2p-2}j_{2p-2}}^{\varepsilon_{2p-2}} e_ie_i^* \ge 0$.
Thus
$$
	\sum_{(\boldsymbol{i},\boldsymbol{j},\boldsymbol{\varepsilon})\sim\pi}
	\ntr U_{i_1 j_1}^{\varepsilon_1}\cdots 
	U_{i_{2p}j_{2p}}^{\varepsilon_{2p}}
	\le
	\sigma^2 
	\sum_{(\boldsymbol{i},\boldsymbol{j},\boldsymbol{\varepsilon})\sim\pi
	\backslash\{\{2p-1,2p\}\}}
	\ntr 
	U_{i_1 j_1}^{\varepsilon_1}\cdots 
	U_{i_{2p-2}j_{2p-2}}^{\varepsilon_{2p-2}},
$$
with equality if $b_{ij}=1$ for all $i,j$. As $\pi\backslash 
\{\{2p-1,2p\}\}$ is again a noncrossing pairing, we can 
iterate this argument to conclude the proof.
\end{proof}

Next, we analyze the effect of a single crossing.

\begin{lem}[Crossing inequality]
\label{lem:cplxcr}
Let $M_1,\ldots,M_4$ be any $n\times n$ matrices with nonnegative entries.
Then we have
$$
	\sum_{i\ge j,k\ge l}
	\sum_{\varepsilon,\delta\in \{1,*\}}
	\ntr U_{ij}^\varepsilon M_1 U_{kl}^{\delta} M_2
	U_{ij}^{!\varepsilon} M_3 U_{kl}^{!\delta} M_4 \le
	\sigma_*^4
	\ntr M_3 M_2 M_1 M_4,
$$
with equality when $b_{ij}=1$ for all $i,j$.
Here we denote $!1:=*$ and $!*:=1$.
\end{lem}

\begin{proof}
Substituting the definition of $U_{ij}$ into the left-hand side yields
\begin{align*}
	&\sum_{i\ge j,k\ge l}
	\sum_{\varepsilon,\delta\in \{1,*\}}
	\ntr U_{ij}^\varepsilon M_1 U_{kl}^{\delta} M_2
	U_{ij}^{!\varepsilon} M_3 U_{kl}^{!\delta} M_4 
	=\\
	&\qquad\qquad\frac{1}{n}
	\sum_{i,j,k,l\in[n]}
	b_{ij}^2 b_{kl}^2
	(M_3)_{il} (M_2)_{lj} (M_1)_{jk} (M_4)_{ki}
\end{align*}
by an explicit computation. The conclusion follows readily.
\end{proof}

We can now prove Theorem \ref{thm:hermmom}.

\begin{proof}[Proof of Theorem \ref{thm:hermmom}]
Fix a pairing $\pi\in\mathrm{P}_2([2p])$. Suppose that $\{1,l\},\{k,m\}$ 
form a crossing $1<k<l<m$. Then Lemma \ref{lem:cplxcr} yields
$$
	\sum_{(\boldsymbol{i},\boldsymbol{j},\boldsymbol{\varepsilon})\sim\pi}
	\ntr U_{i_1 j_1}^{\varepsilon_1}\cdots 
	U_{i_{2p}j_{2p}}^{\varepsilon_{2p}}
	\le
	\sigma_*^4
	\sum_{(\boldsymbol{i},\boldsymbol{j},\boldsymbol{\varepsilon})\sim\pi'}
	\ntr U_{i_1 j_1}^{\varepsilon_1}\cdots 
	U_{i_{2p-4}j_{2p-4}}^{\varepsilon_{2p-4}}.
$$
where $\pi'\in \mathrm{P}_2([2p-4])$ is obtained from $\pi$ by removing 
$\{1,l\},\{k,m\}$ and reordering the remaining elements as
$l+1,\ldots,m-1,k+1,\ldots,l-1,2,\ldots,k-1,m+1,\ldots,2p$. (Here we used 
that any product of matrices $U_{ij}^\varepsilon$ has nonnegative 
elements.)

We can now iterate this procedure until we arrive at a final pairing 
$\pi'$ that is noncrossing. More precisely, given any pairing $\pi$ that
contains at least one crossing, we choose the smallest crossing in the 
lexicographic order and use cyclic permutation of the trace to make its 
smallest element equal to one. We then apply the above inequality. If 
$\pi'$ still contains a crossing, we repeat this procedure until no 
crossings are left. Denote by $\ell(\pi)$ the number of times this
process is repeated until we reach a noncrossing pairing. Then
Lemmas \ref{lem:cplxwick} and \ref{lem:cplxnc} yield
$$
	\mathbf{E}[\ntr X^{2p}] 
	\le
	\sum_{\pi\in \mathrm{P}_2([2p])}
	\sigma^{2p-4\ell(\pi)}
	\sigma_*^{4\ell(\pi)}.
$$
Note that, by construction, the quantity $\ell(\pi)$ is uniquely 
determined by $\pi$ and does not depend on the random matrix $X$.

We now apply precisely the same argument to the random matrix $Y$. 
However, as every entry of $Y$ has equal variance, Lemmas \ref{lem:cplxnc} 
and \ref{lem:cplxcr} ensure that all the above 
inequalities become equalities in this case. In particular, we obtain
$$
	\mathbf{E}[\ntr Y^{2p}] =
	\sum_{\pi\in \mathrm{P}_2([2p])} d^{p-2\ell(\pi)}.
$$
As we assumed $\sigma^2\le d$ and $\sigma_*^2=1$, the conclusion is 
immediate.
\end{proof}

\begin{rem}
An identity of the above form for $\mathbf{E}[\ntr Y^{2p}]$ is classical
in random matrix theory, where it is known as the \emph{genus expansion}
\cite[Theorem 22.12]{NS06}. In particular, the precise combinatorial 
meaning of $\ell(\pi)$ can be understood in terms of the 
genus of the orientable surface obtained by gluing together the edges
of a regular $2p$-gon corresponding to
each pair of $\pi$. However, this combinatorial intepretation is 
completely irrelevant for our purposes: all we used is that the 
inequalities we apply to the nonhomogeneous matrix $X$ become equality for $Y$.
For the real random matrices investigated in the following sections, 
the combinatorial structure is much more delicate while a
comparison argument remains tractable.
\end{rem}

\subsection{Small deviations}
\label{sec:hermsm}

We can now combine Theorem \ref{thm:hermmom} with a method of Ledoux 
\cite[\S 5.2]{Led07} to obtain small deviations inequalities at the 
Tracy-Widom scale. The basis for this method is an accurate estimate on 
the $p$th moment of $Y$ for moderately large $p$. The following can be 
read off from \cite[pp.\ 210--211]{Led07}.

\begin{lem}[Ledoux]
\label{lem:cplxled}
Define $Y$ as in Theorem \ref{thm:hermmom}. 
For all $p\in\mathbb{N}$, we have
$$
	\EE[\ntr Y^{2p}]  \le 
	\frac{1}{p^{3/2}\sqrt{\pi}}
	\bigg(4d+\frac{p(p-1)}{d} \bigg)^p.
$$
\end{lem}

We obtain the following.

\begin{prop}
\label{prop:pchermsm}
For $X$ as in Model \ref{mod:herm} and $0\le\varepsilon\le 1$, we have
$$
	\mathbf{P}\big[
	\|X\|> 2\sqrt{\sigma^2+\sigma_*^2}\,(1+\varepsilon)\big] 
	\le
	\frac{en\sigma_*^2}{\sigma^2}
	e^{-\frac{\sigma^2}{\sigma_*^2}\varepsilon^{3/2}}.
$$
\end{prop}

\begin{proof}
Suppose first that $\sigma_*=1$, and let
$d=\lceil\sigma^2\rceil$.
Using Markov's inequality, 
$\|X\|^{2p}\le n\ntr X^{2p}$, and Theorem \ref{thm:hermmom}, we obtain
$$
	\mathbf{P}\big[\|X\|> 2\sqrt{d}\,(1+\varepsilon)\big] \le
	\frac{n}{(4d)^p}
	\frac{\mathbf{E}[\ntr Y^{2p}]}{(1+\varepsilon)^{2p}}
$$
for all $p\in\mathbb{N}$. Applying Lemma 
\ref{lem:cplxled} yields
$$
	\mathbf{P}\big[\|X\|> 2\sqrt{d}\,(1+\varepsilon)\big] \le
	\frac{1}{\sqrt{\pi}}
	\frac{n}{p^{3/2}}
	e^{-\varepsilon p\log 4 + p^2(p-1)/4d^2}
$$
using that $(1+\varepsilon)^{-2p} \le 4^{-\varepsilon p}$ for
$0<\varepsilon\le 1$. Choosing $p=\lceil 
d\sqrt{2\varepsilon}\rceil$ yields
$$
	\mathbf{P}\big[\|X\|> 2\sqrt{d}\,(1+\varepsilon)\big] \le
	1.3\,
	\frac{n}{d}
	\frac{1}{ (d \varepsilon^{3/2} )^{1/2}}
	e^{-d\varepsilon^{3/2}}
$$
using $p^2(p-1) \le (d\sqrt{2\varepsilon}+1)^2 
d\sqrt{2\varepsilon} \le 2\sqrt{2}\,d^3\varepsilon^{3/2} + 
4d^2+d\sqrt{2}$ and $\sqrt{2}\log 4 -\frac{1}{\sqrt{2}} \ge 1$.

We now consider two cases. First, if $\sigma^2\varepsilon^{3/2} \ge 1$, 
we can estimate
$$
	\mathbf{P}\big[\|X\|> 2\sqrt{d}\,(1+\varepsilon)\big] \le
	\frac{en}{\sigma^2}
	e^{-\sigma^2\varepsilon^{3/2}}
$$
using $e\ge 1.3$ and $d\ge\sigma^2$.  On the other hand, for 
$\sigma^2\varepsilon^{3/2}<1$, we have
$$
	\mathbf{P}\big[\|X\|> 2\sqrt{d}\,(1+\varepsilon)\big] \le
	1\le 
	\frac{en}{\sigma^2}
	e^{-\sigma^2\varepsilon^{3/2}}
$$
as $\sigma^2 \le n\sigma_*^2 = n$. Combining these bounds, we obtain
$$
	\mathbf{P}\big[\|X\|> 2\sqrt{\sigma^2+1}\,(1+\varepsilon)\big] \le
	\frac{en}{\sigma^2}
	e^{-\sigma^2\varepsilon^{3/2}}.
$$
where we used $d\le \sigma^2+1$. This concludes the proof when 
$\sigma_*=1$. For general $\sigma_*$, it suffices to apply the above bound 
to the random matrix $\frac{X}{\sigma_*}$.
\end{proof}

Theorem \ref{thm:hermsm} follows readily.

\begin{proof}[Proof of Theorem \ref{thm:hermsm}]
Proposition \ref{prop:pchermsm} implies
$$
	\mathbf{P}\big[\|X\|>
	2\sqrt{\sigma^2+\sigma_*^2}
	+ 2\sqrt{\sigma^2+\sigma_*^2} \,
	\sigma_*^{4/3}\sigma^{-4/3} t
	\big] 
	\le
	\frac{en\sigma_*^2}{\sigma^2}
	e^{-t^{3/2}}
$$
by setting $\varepsilon = t \sigma_*^{4/3}\sigma^{-4/3}$. We now bound
$$
	\sqrt{\sigma^2+\sigma_*^2} \le
	\sigma + \frac{\sigma_*^2}{2\sigma} \le
	\frac{3}{2}\sigma
$$
to obtain
$$
	\mathbf{P}\big[\|X\|\ge 
	2\sigma + \tfrac{\sigma_*^2}{\sigma}
	+ 3\sigma_*^{4/3}\sigma^{-1/3} t
	\big] 
	\le
	\frac{en\sigma_*^2}{\sigma^2}
	e^{-t^{3/2}}.
$$
This readily implies the conclusion for  
$t\ge\sigma_*^{2/3}\sigma^{-2/3}$. 
On the other hand, as
$$
	1 \le 
	\frac{n\sigma_*^2}{\sigma^2}
	\le
	\frac{en\sigma_*^2}{\sigma^2}
        e^{-t^{3/2}} 
$$
for $t<\sigma_*^{2/3}\sigma^{-2/3}\le 1$, the conclusion holds
trivially in this case.
\end{proof}

\subsection{Large deviations}
\label{sec:hermlg}

The moment estimate of Lemma \ref{lem:cplxled} is not accurate for large 
$p$: indeed, this estimate yields a subexponential bound $\mathbf{E}[\ntr 
Y^{2p}]^{1/2p} =O(p)$ as $p\to\infty$, while the large deviations of the 
norms of Gaussian random matrices are in fact subgaussian. To obtain a 
large deviation inequality, we will instead employ the following 
estimate of Haagerup and Thorbj{\o}rnsen \cite[Eq.\ (3.5)]{HT03}.

\begin{lem}[Haagerup-Thorbj{\o}rnsen]
\label{lem:ht}
Let $Y$ be as in Theorem \ref{thm:hermmom}.
For all $t\ge 0$
$$
	\mathbf{E}[\ntr e^{t Y}] \le
	e^{2\sqrt{d}\,t + t^2/2}.
$$
\end{lem}

We can now prove Theorem \ref{thm:hermlg}.

\begin{proof}[Proof of Theorem \ref{thm:hermlg}]
Suppose first that $\sigma_*=1$, and let $d=\lceil\sigma^2\rceil$.
Note that all odd moments of our random matrices vanish
$\EE[\ntr X^{2p+1}]=\EE[\ntr Y^{2p+1}]=0$ by the symmetry of the Gaussian 
distribution. We can therefore estimate 
$$
	\EE[\ntr e^{t X}] \le
	\EE[\ntr e^{t Y}]
$$
for $t\ge 0$ by Taylor expanding the exponential function and applying
Theorem \ref{thm:hermmom} to the terms of even degree.
As $e^{t\|X\|} \le n \ntr e^{tX} + n\ntr e^{-t X}$, we obtain
$$
	\EE[e^{t\|X\|}] \le
	2n\, e^{2\sqrt{d}\,t+t^2/2}	
$$
using Lemma \ref{lem:ht}.
By Markov's inequality
$$
	\mathbf{P}\big[\|X\| > 2\sqrt{d} + \varepsilon\big] \le
	\frac{\mathbf{E}[e^{t\|X\|}]}{
	e^{(2\sqrt{d} + \varepsilon) t}} \le
	2n\, e^{-\varepsilon t+t^2/2}.
$$
Optimizing over $t\ge 0$ yields
$$
	\mathbf{P}\big[\|X\| > 2\sqrt{\sigma^2+1}+\varepsilon\big] \le
	2n\, e^{-\varepsilon^2/2}
$$
for all $\varepsilon\ge 0$, where we used $d\le \sigma^2+1$.

For general $\sigma_*$, applying the above bound
to the random matrix $\frac{X}{\sigma_*}$ yields
$$
	\mathbf{P}\big[
	\|X\|> 2\sqrt{\sigma^2+\sigma_*^2} + \sigma_*\varepsilon\big] 
	\le
	2n\, e^{-\varepsilon^2/2}
$$
for all $\varepsilon\ge 0$. To conclude, it remains to
use that $\sqrt{\sigma^2+\sigma_*^2} \le \sigma + 
\frac{\sigma_*^2}{2\sigma}\le \sigma+\frac{\sigma_*}{2}$.
\end{proof}

\section{The symmetric case}
\label{sec:real}

We now turn to the case of (real) symmetric matrices. The following model 
is the real symmetric analogue of Model \ref{mod:rectintro}.

\begin{model}
\label{mod:symm}
$X$ is an $n\times n$ real symmetric matrix with $X_{ij}=b_{ij}\xi_{ij}$
for $i\ne j$ and $X_{ii}=\sqrt{2}\,b_{ii}\xi_{ii}$. Here 
$\xi_{ij}=\xi_{ji}$ are 
independent symmetrically distributed real random variables for $i\ge j$ 
with $\mathbf{E}[\xi_{ij}^{2p}]\le
\mathbf{E}[g^{2p}]$ for all $p\in\mathbb{N}$ (here $g\sim N(0,1)$),
and $b_{ij}=b_{ji}\ge 0$ are arbitrary nonnegative scalars.
\end{model}

\begin{rem}
The slightly different scaling of the off-diagonal and 
diagonal entries ensures that $X$ is GOE 
(a real symmetric random matrix whose law is invariant under 
orthogonal conjugation) when $b_{ij}=1$ for all $i,j$ and $\xi_{ij}$ are 
Gaussian.
\end{rem}

In this setting, the parameters $\sigma$ and $\sigma_*$ are defined as in
\eqref{eq:sigmaparm}. The main results of this section are the following 
extremum principle and
tail bounds.

\begin{thm}[Extremum principle]
\label{thm:symmmom}
Define $X$ as in Model \ref{mod:symm}, and assume that
$\sigma_*^2=1$ and that $\sigma^2\le d\in\mathbb{N}$. Then we have
$$
	\mathbf{E}[\ntr X^{2p}] \le \mathbf{E}[\ntr Y^{2p}]
$$
for all $p\in\mathbb{N}$, where $Y$ is the $d\times d$ real symmetric 
matrix whose entries $Y_{ij}=Y_{ji}$ are independent for $i\ge j$ with
$Y_{ij}\sim N(0,1)$ for $i>j$
and $Y_{ii}\sim N(0,2)$.
\end{thm}

\begin{thm}[Small deviations]
\label{thm:symmsm}
For $X$ as in Model \ref{mod:symm}, we have
$$
	\mathbf{P}\big[\|X\|>
	2\sigma + \sigma_*^{4/3}\sigma^{-1/3} t
	\big] 
	\le
	\frac{n\sigma_*^2}{C\sigma^2}
	e^{-Ct^{3/2}}
$$
for every $0\le t\le \frac{\sigma^{4/3}}{\sigma_*^{4/3}}$,
where $C$ is a universal constant.
\end{thm}

\begin{thm}[Large deviations]
\label{thm:symmlg}
For $X$ as in Model \ref{mod:symm}, we have
$$
	\mathbf{P}\big[
	\|X\|> 2\sigma + \sigma_*(1+t)\big] 
	\le
	2n\, e^{-t^2/4}
$$
for every $t\ge 0$.
\end{thm}

The remainder of this section is devoted to the proofs of these results.

\begin{rem}
The scale of the fluctuations in the tail bound of Theorem 
\ref{thm:symmlg} is $\sqrt{2}\sigma_*$, while the scale of the 
fluctuations
in Theorem \ref{thm:hermlg} is only $\sigma_*$. The larger tail in the 
real symmetric case is necessary, however: for example, Theorem 
\ref{thm:symmlg} implies that
$$
	\|X\| \le 2\sigma + (1+o(1))\,2\sigma_*\sqrt{\log n}
	\quad\text{w.h.p.}
$$
as $n\to\infty$, while we have
$$
	\|X\| \ge \max_i X_{ii} \ge (1+o(1))\,2\sigma_*\sqrt{\log n}
	\quad\text{w.h.p.}
$$
whenever $b_{ii}=\sigma_*$ and $\xi_{ii}\sim N(0,1)$ for all $i$ (as then 
$X_{ii}=\sqrt{2}\sigma_*\xi_{ii}$). Moreover, it follows from
\cite[Corollary 3.2]{LT91} that the tail bounds of both
Theorems \ref{thm:hermlg} and \ref{thm:symmlg} cannot be improved in 
general for $t\to\infty$ (even in models where $b_{ii}=0$ for all $i$).
\end{rem}

\subsection{Extremum principle}

Let $(g_{ij})_{i\ge j}$ be i.i.d.\ $N(0,1)$ random variables. We first 
show that there is no loss in assuming that $\xi_{ij}=g_{ij}$ in Model 
\ref{mod:symm}.

\begin{lem}
\label{lem:gcomp}
Let $X$ be as in Model \ref{mod:symm}, and let $\tilde X$ be defined
as $X$ where $\xi_{ij}$ is replaced by $g_{ij}$. Then
$\mathbf{E}[\ntr X^{2p}]\le\mathbf{E}[\ntr \tilde X^{2p}]$ for all 
$p\in\mathbb{N}$.
\end{lem}

\begin{proof}
Note that $\ntr X^{2p}$ is a polynomial of 
$(X_{ij})_{i\ge j}$ with nonnegative coefficients, and that the 
expectation of each monomial is either zero or a product of terms of the 
form $\mathbf{E}[X_{ij}^{2k}]$ ($k\in\mathbb{N}$)
as $X_{ij}$ are symmetrically distributed and independent.
The conclusion follows from the assumption 
that $\mathbf{E}[\xi_{ij}^{2k}]\le\mathbf{E}[g_{ij}^{2k}]$ for all 
$k\in\mathbb{N}$.
\end{proof}

We will therefore assume in the remainder of the proof of Theorem 
\ref{thm:symmmom} that $X$ is defined as in 
Model \ref{mod:symm} with $\xi_{ij}=g_{ij}$. Then we define
$$
	H_{ij} := b_{ij}(e_ie_j^*+e_je_i^*),\qquad\qquad
	H_{ii} := \sqrt{2}\, b_{ii} e_ie_i^*
$$
for $i>j$, and represent the random matrix $X$ as
$$
	X = \sum_{i\ge j} g_{ij} H_{ij}.
$$
We can now compute the moments of $X$ as follows.

\begin{lem}[Wick formula]
\label{lem:symmwick}
For any $p\in\mathbb{N}$, we have
$$
	\EE[\ntr X^{2p}] =
	\sum_{\pi\in \mathrm{P}_2([2p])}
	\ntr H(\pi)
$$
with
$$
	H(\pi) :=
	\sum_{(\boldsymbol{i},\boldsymbol{j})\sim\pi}
	H_{i_1 j_1}\cdots 
	H_{i_{2p}j_{2p}}.
$$
Here $\boldsymbol{i},\boldsymbol{j}\in [n]^{2p}$, and
$(\boldsymbol{i},\boldsymbol{j})\sim\pi$ denotes
$i_k\ge j_k$ and
$i_k=i_l,j_k=j_l$ for
$\{k,l\}\in\pi$.
\end{lem}

\begin{proof}
The proof is identical to that of Lemma \ref{lem:cplxwick}.
\end{proof}

The challenge in the real case is that the effect of a crossing is much 
more complicated
than in the complex case (Lemma \ref{lem:cplxcr}): we 
will have to distinguish in the analysis between two different types of 
crossings that contribute in a different way to the Wick formula. In order 
to do so efficiently, it will turn out to be necessary to work with the 
following upper bound on the trace moments.

\begin{lem}
\label{lem:cpi}
For any $k\in\mathbb{N}$ and $\pi\in\mathrm{P}_2([2k])$, define
$$
	C(\pi) := \max_r (H(\pi))_{rr}.
$$
Then
$$
	\mathbf{E}[\ntr X^{2p}] \le
	\sum_{\pi\in \mathrm{P}_2([2p])} C(\pi).
$$
Moreover, when $b_{ij}=1$ for all $i,j$, the matrix
$H(\pi)$ is a multiple of the identity matrix for every $\pi$ and thus the 
above inequality holds with equality.
\end{lem}

\begin{proof}
The inequality is immediate from Lemma \ref{lem:symmwick}.
It remains to show that $H(\pi)$
is a multiple of the identity matrix 
when $b_{ij}=1$ for all $i,j$.
To this end, note that we may write
$H(\pi)=\mathbf{E}[X_1\cdots X_{2k}]$, where
$X_k=X_l$ is an independent copy of $X$ for each $\{k,l\}\in\pi$. When 
$b_{ij}\equiv 1$, $X$ is a GOE matrix and its distribution is therefore 
orthogonally invariant. Thus
$O\mathbf{E}[X_1\cdots X_{2k}]O^* =
\mathbf{E}[OX_1O^*\cdots OX_{2k}O^*] = \mathbf{E}[X_1\cdots X_{2k}]$
for every $O\in O(n)$, from which the conclusion follows directly.
\end{proof}

We will now analyze the quantity $C(\pi)$. We first consider noncrossing 
pairings.

\begin{lem}
\label{lem:symmnc}
For any $p\in\mathbb{N}$ and noncrossing pairing 
$\pi\in\mathrm{NC}_2([2p])$, we have
$$
	C(\pi) \le \tilde\sigma^{2p},\qquad\quad
	\tilde\sigma^2 := \max_i \bigg(\sum_j b_{ij}^2 + b_{ii}^2\bigg).
$$
Moreover, equality holds when $b_{ij}=1$ for all $i,j$.
\end{lem}

\begin{proof}
As in the proof of Lemma \ref{lem:cplxnc}, there exists
$\{k,k+1\}\in\pi$. Then
$$
	C(\pi) =
	\max_r
	\sum_{(\boldsymbol{i},\boldsymbol{j})\sim\pi\backslash\{\{k,k+1\}\}} 
	(
	H_{i_1j_1}\cdots
	H_{i_{k-1}j_{k-1}}
	\Sigma
	H_{i_{k+2}j_{k+2}}\cdots
        H_{i_{2p}j_{2p}}
	)_{rr},
$$
where
$$
	\Sigma :=
	\sum_{i\ge j} H_{ij}^2 =
	\sum_{i}
	\bigg(\sum_j b_{ij}^2 + b_{ii}^2\bigg) e_i e_i^*.
$$
But note that $(H_{i_1j_1}\cdots H_{i_{k-1}j_{k-1}} e_i e_i^*
H_{i_{k+2}j_{k+2}}\cdots
H_{i_{2k}j_{2k}})_{rr}\ge 0$
as all $H_{ij}$ have nonnegative entries. We can therefore readily 
estimate
$$
	C(\pi) \le \tilde\sigma^2\,
	C(\pi\backslash\{\{k,k+1\}\}),
$$
with equality if $b_{ij}=1$ for all $i,j$.
As $\pi\backslash\{\{k,k+1\}\}$ is again noncrossing, we can iterate this 
argument to conclude the proof.
\end{proof}

We now consider the effect of a crossing. We will need the
following identities, which follow from somewhat tedious but
straightforward computations.

\begin{lem}[Crossing identities]
\label{lem:symmcrid}
The following hold.
\begin{enumerate}[a.]
\item For any matrices $M_1,M_2,M_3$, we have
\begin{multline*}
	\sum_{i\ge j} \sum_{k\ge l}
	H_{ij} M_1 H_{kl} M_2 H_{ij} M_3 H_{kl} = 
	\sum_{i,j,k,l} b_{ij}^2 b_{kl}^2
	\big\{(M_1)_{jk} (M_2)_{lj} (M_3)_{il} + \mbox{}\\
	(M_1)_{jl} (M_2)_{kj} (M_3)_{il} +
	(M_1)_{jk} (M_2)_{li} (M_3)_{jl} +
	(M_1)_{jl} (M_2)_{ki} (M_3)_{jl}\big\}
	e_i e_k^*.
\end{multline*}
\item For any matrix $M$, we have
$$
	\sum_{i\ge j} \mathrm{Tr}[M H_{ij}] H_{ij} =
	\sum_{i,j} b_{ij}^2 (M_{ij} + M_{ji})e_i e_j^*.
$$
\end{enumerate}
\end{lem}

\begin{proof}
We readily compute that for any matrix $M$, we have
$$
	\sum_{i\ge j} H_{ij} M H_{ij} =
	\sum_{i,j} b_{ij}^2 
	(M_{jj} e_i e_i^*  + M_{ji} e_i e_j^* ).
$$
The first part of the statement follows by applying this identity 
twice. The second part of the statement is again a straightforward 
computation.
\end{proof}

To proceed, we must distinguish between two types of crossings that play a 
fundamentally different role in the argument.

\begin{defn}
Let $\pi\in \mathrm{P}_2([2p])$, and fix a crossing 
$\{i,k\},\{j,l\}\in\pi$ such that $i<j<k<l$.
Then this crossing is said to be of \emph{type I} if 
there exists $\{a,b\}\in\pi$ with
$a\in (i,j)\cup (k,l)$ and $b\not\in (i,j)\cup (k,l)$, an is said to be
of \emph{type II} otherwise.
\end{defn}

We consider each type of crossing in turn.

\begin{lem}[Type I crossing]
\label{lem:symmi}
Let $p\in\mathbb{N}$, $\pi\in\mathrm{P}_2([2p])$, 
$\{a,c\},\{b,d\},\{e,f\}\in\pi$ with $a<b<c<d$ and
$e\in (a,b)\cup (c,d)$, $f\not\in (a,b)\cup (c,d)$.
Then there exist pairings $\pi_1,\pi_2,\pi_3\in \mathrm{P}_2([2p-4])$ 
and $\pi_4,\pi_5\in \mathrm{P}_2([2p-6])$, whose definition depends only 
on $\pi,a,b,c,d,e,f$ (and not on the matrix $X$), so that\footnote{%
Here $C(\pi_4)=C(\varnothing):=1$ if $2p-6=0$. The analogous convention 
will be used in the sequel.}
$$
	C(\pi) \le 
	\sigma_*^4 \{C(\pi_1)+C(\pi_2)+C(\pi_3)\}+
	\sigma_*^6\, \{C(\pi_4)+C(\pi_5)\}.
$$
Moreover, equality holds when $b_{ij}=1$ for all $i,j$.
\end{lem}

\begin{proof}
For any $M_1,M_2,M_3$ with nonnegative entries, 
Lemma \ref{lem:symmcrid}(a) yields
\begin{multline*}
	\sum_{i\ge j} \sum_{k\ge l}
	(H_{ij} M_1 H_{kl} M_2 H_{ij} M_3 H_{kl})_{rs} \\ \le
	\sigma_*^4
	\big(M_3M_2M_1 +
	M_3 M_1^* M_2^* +
	M_2^*M_3^*M_1 +
	\mathrm{Tr}[M_1^*M_3] M_2^*\big)_{rs}
\end{multline*}
with equality if $b_{ij}=1$ for all $i,j$. We can therefore write
\begin{multline*}
	C(\pi) \le
	\sigma_*^4
	\max_r
	\sum_{(\boldsymbol{i},\boldsymbol{j})\sim\pi
	\backslash\{\{a,c\},\{b,d\}\}}
	\big\{
	(M_0M_3M_2M_1M_4)_{rr}
	+
	(M_0M_3 M_1^* M_2^*M_4)_{rr}  \\ 
	+ (M_0M_2^*M_2^*M_1M_4)_{rr} +
	\mathrm{Tr}[M_1^*M_3] (M_0M_2^*M_4)_{rr}\big\}
\end{multline*}
with equality if $b_{ij}=1$ for all $i,j$, where
\begin{align*}
	M_0 &:= H_{i_1j_1}\cdots H_{i_{a-1}j_{a-1}},\\
	M_1 &:= H_{i_{a+1}j_{a+1}}\cdots H_{i_{b-1}j_{b-1}},\\
	M_2 &:= H_{i_{b+1}j_{b+1}}\cdots H_{i_{c-1}j_{c-1}},\\
	M_3 &:= H_{i_{c+1}j_{c+1}}\cdots H_{i_{d-1}j_{d-1}},\\
	M_4 &:= H_{i_{d+1}j_{d+1}}\cdots H_{i_{2p}j_{2p}}.
\end{align*}
For simplicity, we consider in the remainder of the proof the case that 
$c<e<d$ and $f>d$; the proof of the five other possible cases is 
completely analogous. Then we can use Lemma \ref{lem:symmcrid}(b) to estimate
\begin{align*}
	&\sum_{(\boldsymbol{i},\boldsymbol{j})\sim\pi
        \backslash\{\{a,c\},\{b,d\}\}}
	\mathrm{Tr}[M_1^*M_3] (M_0M_2^*M_4)_{rr} \\
	&=
	\sum_{(\boldsymbol{i},\boldsymbol{j})\sim\pi
        \backslash\{\{a,c\},\{b,d\},\{e,f\}\}}
	\sum_{i\ge j}
	\mathrm{Tr}[M_1^*M_3^- H_{ij} M_3^+] (M_0M_2^*M_4^- H_{ij} 
	M_4^+)_{rr}
	\\
	&
	\le \sigma_*^2
	\sum_{(\boldsymbol{i},\boldsymbol{j})\sim\pi
        \backslash\{\{a,c\},\{b,d\},\{e,f\}\}}
	\big\{
	(M_0M_2^*M_4^- 
	M_3^+M_1^*M_3^- 
	M_4^+)_{rr} \\
	&
	\qquad
	\qquad
	\qquad
	\qquad
	\qquad
	\qquad
	\qquad
	+
	(M_0M_2^*M_4^- 
	M_3^{-*} M_1M_3^{+*}
	M_4^+)_{rr}\big\}
\end{align*}
with equality if $b_{ij}=1$ for all $i,j$, where 
\begin{align*}
	M_3^- &:= H_{i_{c+1}j_{c+1}}\cdots H_{i_{e-1}j_{e-1}},\\
	M_3^+ &:= H_{i_{e+1}j_{e+1}}\cdots H_{i_{d-1}j_{d-1}},\\
	M_4^- &:= H_{i_{d+1}j_{d+1}}\cdots H_{i_{f-1}j_{f-1}},\\
	M_4^+ &:= H_{i_{f+1}j_{f+1}}\cdots H_{i_{2p}j_{2p}}.
\end{align*}
From the above, we can readily read off the existence of
$\pi_1,\pi_2,\pi_3\in \mathrm{P}_2([2p-4])$
and $\pi_4,\pi_5\in \mathrm{P}_2([2p-6])$, whose definition depends only
on $\pi,a,b,c,d,e,f$, so that
$$
	C(\pi) \le
	\max_r \big\{
	\sigma_*^4\big(H(\pi_1)+H(\pi_2)+H(\pi_3)\big)_{rr} +
	\sigma_*^6\big(H(\pi_4)+H(\pi_5)\big)_{rr}
	\big\}
$$
with equality if $b_{ij}=1$ for all $i,j$, where $H(\pi)$ was defined in 
Lemma \ref{lem:symmwick}. The inequality in the statement follows 
immediately by using $(H(\pi))_{rr}\le C(\pi)$ for all $r,\pi$.
On the other hand, when $b_{ij}=1$ for all $i,j$, Lemma 
\ref{lem:cpi} yields $H(\pi)=C(\pi)\mathbf{1}$ for 
all $\pi$, so that the inequality in the statement holds 
with equality.
\end{proof}

When Lemma \ref{lem:symmcrid}(a) is applied as in the proof of Lemma 
\ref{lem:symmi}, it creates a term that involves an unnormalized trace. 
For crossings of type I, we can subsequently apply Lemma 
\ref{lem:symmcrid}(b) to eliminate the trace and regain a term of the form 
$H(\pi)$. This is not possible, however, for crossings of type II, which 
would cause the trace to factor out of the expression. As this trace is 
unnormalized, that would ultimately yield a dimension-dependent factor in 
the final bound.

To avoid this, we must apply Lemma \ref{lem:symmcrid}(a) in a different 
manner to crossings of type II. It is here that we rely on the fact that 
we work with the quantity $C(\pi)$, rather than the smaller quantity $\ntr 
H(\pi)$ that appears in the Wick formula itself. (This is in contrast to 
the statement and proof of Lemma \ref{lem:symmi}, which would work equally 
well if we replace $C(\pi)$ by $\ntr H(\pi)$ throughout.)

\begin{lem}[Type II crossing]
\label{lem:symmii}
Let $p\in\mathbb{N}$, $\pi\in\mathrm{P}_2([2p])$, and let
$\{a,c\},\{b,d\}\in\pi$ with $a<b<c<d$ be a crossing of type II.
Then there exist pairings $\pi_1,\pi_2,\pi_3\in \mathrm{P}_2([2p-4])$ 
and $\pi_4\in \mathrm{P}_2([b-a+d-c-2])$,
$\pi_5\in \mathrm{P}_2([2p-2 -b+a -d+c])$, whose definition depends only 
on $\pi,a,b,c,d$ (and not on the matrix $X$), so that
$$
	C(\pi) \le 
	\sigma_*^4 \{C(\pi_1)+C(\pi_2)+C(\pi_3)\}+
	\sigma^2\sigma_*^2\, C(\pi_4)\, C(\pi_5).
$$
Moreover, equality holds when $b_{ij}=1$ for all $i,j$.
\end{lem}

\begin{proof}
Define $M_0,\ldots,M_4$ (which depend on $\boldsymbol{i},\boldsymbol{j}$)
and $\pi_1,\pi_2,\pi_3$ as in the proof of Lemma \ref{lem:symmi}.
Then we can estimate using Lemma \ref{lem:symmcrid}(a) 
\begin{multline*}
	C(\pi) \le
	\max_r\bigg\{
	\sigma_*^4\big(H(\pi_1)+H(\pi_2)+H(\pi_3)\big)_{rr} + \\
	\sigma_*^2
	\sum_{(\boldsymbol{i},\boldsymbol{j})\sim\pi
	\backslash\{\{a,c\},\{b,d\}\}}
	\sum_{k,l} b_{kl}^2
	(M_0M_2^*)_{rk}
	(M_1^*M_3)_{ll}
	(M_4)_{kr}
	\bigg\},
\end{multline*}
with equality if $b_{ij}=1$ for all $i,j$. Note that we applied Lemma 
\ref{lem:symmcrid}(a) here exactly as in the proof of Lemma 
\ref{lem:symmi}, except that we only upper bounded $b_{ij}^2\le\sigma_*^2$
in the last term of Lemma
\ref{lem:symmcrid}(a) and avoided upper bounding $b_{kl}^2$ as well.

To proceed more efficiently in the present setting, we 
will now use the special structure of type II crossings.
By the type II assumption, we can decompose
$\pi\backslash\{\{a,c\},\{b,d\}\}=\pi_I\cup\pi_J$ where
$\pi_I\in\mathrm{P}_2(I)$, $\pi_J\in\mathrm{P}_2(J)$ with
$$
	I:=(a,b)\cup(c,d),\qquad\quad
	J:=[1,a)\cup (b,c)\cup (d,2p].
$$
Denote by $\boldsymbol{i}_I := (i_s)_{s\in I}$, with the analogous
notation for $\boldsymbol{j}$ and $J$. Then 
\begin{align*}
	&\sum_{(\boldsymbol{i},\boldsymbol{j})\sim\pi
	\backslash\{\{a,c\},\{b,d\}\}}
	\sum_{k,l} b_{kl}^2
	(M_0M_2^*)_{rk}
	(M_1^*M_3)_{ll}
	(M_4)_{kr}
\\
	&\qquad=
	\sum_{k} 
	\Bigg(
	\sum_{(\boldsymbol{i}_J,\boldsymbol{j}_J)\sim\pi_J}
	(M_0M_2^*)_{rk}
	(M_4)_{kr}
	\Bigg)
	\Bigg(
	\sum_l
	b_{kl}^2
	\sum_{(\boldsymbol{i}_I,\boldsymbol{j}_I)\sim\pi_I}
	(M_1^*M_3)_{ll}
	\Bigg)
\\
	&\qquad\le
	\sigma^2
	\Bigg(
	\sum_{(\boldsymbol{i}_J,\boldsymbol{j}_J)\sim\pi_J}
	(M_0M_2^*M_4)_{rr}
	\Bigg)
	\max_l
	\Bigg(
	\sum_{(\boldsymbol{i}_I,\boldsymbol{j}_I)\sim\pi_I}
	(M_1^*M_3)_{ll}
	\Bigg),
\end{align*}
with equality if $b_{ij}=1$ for all $i,j$ (for the equality case, we used 
the second part of Lemma \ref{lem:cpi}). From the above, we can readily 
read off the existence of pairings
$\pi_4\in \mathrm{P}_2([b-a+d-c-2])$ and
$\pi_5\in \mathrm{P}_2([2p-2 -b+a -d+c])$, whose definition depends only 
on $\pi,a,b,c,d$, such that
$$
	C(\pi) \le
	\max_r\big\{
	\sigma_*^4\big(H(\pi_1)+H(\pi_2)+H(\pi_3)\big)_{rr} +
	\sigma^2\sigma_*^2
	C(\pi_4) (H(\pi_5))_{rr}\big\}
$$
with equality if $b_{ij}=1$ for all $i,j$. The conclusion now follows by
precisely the same argument as at the end of the proof of Lemma 
\ref{lem:symmi}.
\end{proof}

We can now conclude the proof of Theorem \ref{thm:symmmom}.

\begin{proof}[Proof of Theorem \ref{thm:symmmom}]
We first estimate $\mathbf{E}[\ntr X^{2p}]$ as in Lemma 
\ref{lem:cpi}. We then repeatedly apply the following to each quantity 
$C(\pi)$ in the resulting expression:
\begin{enumerate}[$\bullet$]
\itemsep\medskipamount
\item If $\pi$ contains a crossing, we apply either Lemma \ref{lem:symmi}
or Lemma \ref{lem:symmii} to the smallest crossing in the lexicographic 
order, depending on whether that crossing is of type I or type II, 
respectively. In the former case, we choose the smallest 
pair $\{e,f\}$ in the lexicographic order that satisfies the assumption of 
Lemma \ref{lem:symmi}.
\item If $\pi$ is noncrossing, we apply Lemma \ref{lem:symmnc}.
\end{enumerate}
This algorithm gives rise to an inequality of the form
$$
	\mathbf{E}[\ntr X^{2p}] \le
	\sum_{k,l\ge 0:k+l\le p}
	\varkappa_p(k,l)
	\tilde\sigma^{2k}\sigma^{2l}\sigma_*^{2p-2k-2l},
$$
where the coefficients $\varkappa_p(k,l)\in\mathbb{Z}_+$ are independent 
of the matrix $X$.

Moreover, as all the inequalities used above become equalities when 
$b_{ij}=1$ for all $i,j$, we can apply the same algorithm to compute
$$
	\mathbf{E}[\ntr Y^{2p}] =
	\sum_{k,l\ge 0:k+l\le p}
	\varkappa_p(k,l)
	(d+1)^{k}d^{l}.
$$
The conclusion readily follows from the assumptions that
$\sigma_*=1$ and $\sigma^2\le d$, and as we can estimate
$\tilde\sigma^2 \le \sigma^2+1 \le d+1$ using $\sigma_*=1$.
\end{proof}

\subsection{Small deviations}

With the extremum principle in hand, we can now repeat the arguments of
section \ref{sec:hermsm} in the symmetric case. The main difference is 
that now $Y$ is a GOE matrix rather than a GUE matrix. However, a 
suitable moment estimate for GOE matrices was also obtained by Ledoux
\cite[Theorem 8]{Led09}.

\begin{lem}[Ledoux]
\label{lem:symmled}
Define $Y$ as in Theorem \ref{thm:symmmom}. Then for
$p\ge d^{2/3}$, we have
$$
	\mathbf{E}[\ntr Y^{2p}] \lesssim
	\frac{1}{d}
	(4d)^p
	\bigg(
	1+\frac{p^2}{d^2}
	\bigg)^{2p}.
$$
\end{lem}

This yields the following.

\begin{prop}
\label{prop:ledsymm}
For $X$ as in Model \ref{mod:symm}, we have
$$
	\mathbf{P}\big[
	\|X\| > 2\sqrt{\sigma^2+\sigma_*^2}\,(1+\varepsilon)
	\big]
	\le
	\frac{n\sigma_*^2}{C\sigma^2} e^{-\frac{1}{4}\frac{\sigma^2}{\sigma_*^2}
	\varepsilon^{3/2}}
$$
for all $0\le\varepsilon\le 1$, where $C$ is a universal constant.
\end{prop}

\begin{proof}
Suppose first that $\sigma_*=1$, and let $d=\lceil\sigma^2\rceil$.
Using Markov's inequality, $\|X\|^{2p}\le n\ntr X^{2p}$, 
Theorem \ref{thm:symmmom}, and Lemma \ref{lem:symmled}, we obtain
$$
	\mathbf{P}\big[\|X\|>
	2\sqrt{d}\,(1+\varepsilon)\big]
	\lesssim
	\frac{n}{d} (1+\varepsilon)^{-2p}
	\bigg(1+\frac{p^2}{d^2}\bigg)^{2p}
$$
for all $p\ge d^{2/3}$. In particular, $(1+\varepsilon)^{-2p}\le 
e^{-\varepsilon p}$ for $0<\varepsilon\le 1$ and $1+x\le e^x$ yield
$$
	\mathbf{P}\big[\|X\|>
	2\sqrt{d}\,(1+\varepsilon)\big]
	\lesssim
	\frac{n}{d}\, e^{-\varepsilon p + 2p^3/d^2}
$$
for all $p\ge d^{2/3}$.

We now consider two cases. If $\varepsilon\ge 16d^{-2/3}$, choose
$p=\lfloor \frac{1}{2}d\sqrt{\varepsilon}\rfloor\ge d^{2/3}$ to obtain
$$
	\mathbf{P}\big[\|X\|>
	2\sqrt{d}\,(1+\varepsilon)\big]
	\lesssim
	\frac{n}{d}\, e^{-\frac{1}{4}d\varepsilon^{3/2}
	+\varepsilon}
	\le
	\frac{en}{d}\, e^{-\frac{1}{4}d\varepsilon^{3/2}}.
$$
On the other hand, if $\varepsilon<16d^{-2/3}$ we can estimate
$$
	\mathbf{P}\big[\|X\|>
	2\sqrt{d}\,(1+\varepsilon)\big]
	\le
	1
	\lesssim
	\frac{n}{d}\,
	e^{-\frac{1}{4}d\varepsilon^{3/2}}
$$
as $\sigma^2\le n\sigma_*^2=n$ implies $\frac{n}{d}\ge 1$.

Combining the above bounds and using $\sigma^2\le d\le \sigma^2+1$ yields
$$
	\mathbf{P}\big[\|X\|>
	2\sqrt{\sigma^2+1}\,(1+\varepsilon)\big]
	\lesssim
	\frac{n}{\sigma^2}\, e^{-\frac{1}{4}\sigma^2\varepsilon^{3/2}}
$$
for $0\le\varepsilon\le 1$. This concludes the proof when $\sigma_*=1$. 
For general $\sigma_*$, it suffices to apply the above bound to the
random matrix $\frac{X}{\sigma_*}$.
\end{proof}

The rest of the proof of Theorem \ref{thm:symmsm} is now identical to that 
of Theorem \ref{thm:hermsm}.

\subsection{Large deviations}

To obtain a large deviations estimate, we proceed as in section 
\ref{sec:hermlg}. To this end, we require an analogue of Lemma 
\ref{lem:ht} for GOE matrices.

\begin{lem}
\label{lem:goeld}
Let $Y$ be as in Theorem \ref{thm:symmmom}.
For all $t\ge 0$
$$
	\mathbf{E}[\ntr e^{t Y}] \le
	e^{2\sqrt{d}\,t + t^2}.
$$
\end{lem}

\begin{proof}
We clearly have $\ntr e^{tY}\le e^{t\lambda_1(Y)}$, where $\lambda_1(Y)$ 
denotes the largest eigenvalue of $Y$. It is shown in \cite[Theorem 
2.11]{DS01} that $\mathbf{E}[\lambda_1(Y)]\le 2\sqrt{d}$.
On the other hand, the Gaussian log-Sobolev inequality implies
by \cite[eq.\ (5.8)]{Led01} and \cite[\S 2.2]{DS01} that
$\lambda_1(Y)$ is a $\sqrt{2}$-subgaussian random variable, that is, that
$$
	\mathbf{E}\big[e^{t\{\lambda_1(Y)-\mathbf{E}[\lambda_1(Y)]\}}\big]
	\le e^{t^2}
$$
for all $t$. Combining these facts yields the conclusion.
\end{proof}

The rest of the proof of Theorem \ref{thm:symmlg} is now identical to that 
of Theorem \ref{thm:hermlg}.

\section{The independent case}
\label{sec:rect}

The aim of this section is to prove the results for the independent entry 
Model~\ref{mod:rectintro} that were formulated in the introduction. The 
main result of this section is the following, where we 
recall that $\sigma_1,\sigma_2,\sigma_*$ were defined in 
\eqref{eq:sigmarect}.

\begin{thm}[Extremum principle]
\label{thm:rectmom}
Define $X$ as in Model \ref{mod:rectintro}. Assume that
$\sigma_*^2=1$, and that $\sigma_1^2\le d_1\in\mathbb{N}$ and
$\sigma_2^2\le d_2\in\mathbb{N}$. Then
$$
	\mathbf{E}[\ntr (XX^*)^{p}] \le \mathbf{E}[\ntr (YY^*)^{p}]
$$
for all $p\in\mathbb{N}$, where $Y$ is the $d_1\times d_2$  matrix with
independent entries $Y_{ij}\sim N(0,1)$.
\end{thm}

Theorem \ref{thm:rectext} and Corollary \ref{cor:rexn} follow immediately 
from Theorem \ref{thm:rectmom}. To deduce Corollary \ref{cor:rexu}, note 
that the block-diagonal matrix $\tilde Y$ as illustrated in 
Figure \ref{fig:blockd} has $\frac{n}{d_1}$ blocks, each of which is an 
independent copy of $Y$. Therefore 
$$
	\mathbf{E}[\tr (XX^*)^p] \le 
	\frac{n}{d_1}\,\mathbf{E}[\tr (YY^*)^p]
	=\mathbf{E}[\tr (\tilde Y\tilde Y^*)^p]
$$
is merely another reformulation of Theorem \ref{thm:rectmom}.

Theorem \ref{thm:rectmom} is proved in section \ref{sec:pfrectext} below, 
while Theorems \ref{thm:smrect} and \ref{thm:lgrect} will be proved in 
sections \ref{sec:pfrectsm} and \ref{sec:pfrectlg}, respectively.

\subsection{Extremum principle}
\label{sec:pfrectext}

Let $(g_{ij})$ be independent $N(0,1)$ random variables. By repeating the 
proof of Lemma \ref{lem:gcomp} verbatim, we may assume without loss 
of generality in the remainder of this section that $X$ is defined as in
Model \ref{mod:rectintro} with $\xi_{ij}=g_{ij}$. Defining
$E_{ij} := b_{ij}e_ie_j^*$ for $i\le n$, $j\le m$, we can then write
$$
	X = \sum_{i,j} g_{ij} E_{ij}.
$$
The moments of $X$ are computed as follows.

\begin{lem}[Wick formula]
\label{lem:rectwick}
For any $p\in\mathbb{N}$, we have
$$
	\mathbf{E}[\ntr (XX^*)^{p}] = 
	\sum_{\pi\in\mathrm{P}_2([2p])}
	\ntr E(\pi)
$$
with
$$
	E(\pi) := \sum_{(\boldsymbol{i},\boldsymbol{j})\sim\pi}
	E_{i_1j_1} E_{i_2j_2}^* \cdots
	E_{i_{2p-1}j_{2p-1}} E_{i_{2p}j_{2p}}^*.
$$
Here $\boldsymbol{i}\in[n]^{2p},\boldsymbol{j}\in[m]^{2p}$, and
$(\boldsymbol{i},\boldsymbol{j})\sim\pi$ denotes
$i_k=i_l,j_k=j_l$ for $\{k,l\}\in\pi$.
\end{lem}

\begin{proof}
The proof is identical to that of Lemma \ref{lem:cplxwick}.
\end{proof}

The main steps of the proof in this setting are similar to those in the
real symmetric case, which we follow with the appropriate 
modifications. We begin by upper bounding the trace moments by the maximal 
diagonal entry.

\begin{lem}
\label{lem:dpi}
For any $k\in\mathbb{N}$ and $\pi\in\mathrm{P}_2([2k])$, define
$$
	D(\pi) := \max_r (E(\pi))_{rr}.
$$
Then 
$$
	\mathbf{E}[\ntr (XX^*)^{p}] \le
	\sum_{\pi\in\mathrm{P}_2([2p])} D(\pi).
$$
Moreover, when $b_{ij}=1$ for all $i,j$, the matrix $E(\pi)$ is a multiple 
of the identity matrix for every $\pi$ and thus the above inequality holds 
with equality.
\end{lem}

\begin{proof}
The proof is identical to that of Lemma \ref{lem:cpi}, where we use
that when $b_{ij}=1$ the distributions of $X$ and $OX$ coincide for every
$O\in O(n)$.
\end{proof}

We first consider noncrossing pairings.

\begin{lem}
\label{lem:rectnc}
For any $p\in\mathbb{N}$ and noncrossing pairing 
$\pi\in\mathrm{NC}_2([2p])$, we have
$$
	D(\pi) \le \sigma_1^{2\ell(\pi)}\sigma_2^{2(p-\ell(\pi))}
$$
where
$$
	\ell(\pi):=|\{\{i,j\}\in\pi:i\wedge j\text{ is even},~
	i\vee j\text{ is odd}\}|.
$$
Moreover, equality holds when $b_{ij}=1$ for all $i,j$.
\end{lem}

\begin{proof}
As in the proof of Lemma \ref{lem:cplxnc}, there exists
$\{k,k+1\}\in\pi$. If $k$ is even, then
$$
        D(\pi) =
        \max_r
        \sum_{(\boldsymbol{i},\boldsymbol{j})\sim\pi\backslash\{\{k,k+1\}\}}
        (
        E_{i_1j_1}E_{i_2j_2}^*\cdots
        E_{i_{k-1}j_{k-1}}
        \Sigma
        E_{i_{k+2}j_{k+2}}^*
	\cdots
        E_{i_{2k}j_{2k}}^*
        )_{rr},
$$
where
$$
        \Sigma =
        \sum_{i,j} E_{ij}^*E_{ij} =
	\sum_{j} \bigg(\sum_i b_{ij}^2\bigg) e_j e_j^*.
$$
As all $E_{ij}$ have nonnegative entries, we readily estimate
$$
        D(\pi) \le \sigma_1^2\,
        D(\pi\backslash\{\{k,k+1\}\})
$$
with equality if $b_{ij}=1$ for all $i,j$.

On the other hand, if $k$ is odd, we can repeat the same procedure with
$$
        \Sigma =
        \sum_{i,j} E_{ij}E_{ij}^* =
	\sum_{i} \bigg(\sum_j b_{ij}^2\bigg) e_i e_i^*,
$$
and we obtain
$$
        D(\pi) \le \sigma_2^2\,
        D(\pi\backslash\{\{k,k+1\}\})
$$
with equality if $b_{ij}=1$ for all $i,j$.

As $\pi\backslash\{\{k,k+1\}\}$ is again a noncrossing pairing, and as 
removing a consecutive pair from $\pi$ doesn't change the parity of the 
indices of the remaining pairs, we can iterate the above argument
to conclude the proof.
\end{proof}

Before we can analyze the effect of a crossing, we must first obtain the 
appropriate crossing identities. Let us emphasize that while a noncrossing 
pairing can only pair even indices with odd indices, a pairing that 
contains crossings can also pair even indices with each other and odd 
indices with each other. The appropriate crossing identity depends on the 
nature of the pairs in the crossing. Here, $M\leent N$ denotes entrywise
inequality of matrices, i.e., $M_{ij}\le N_{ij}$ for all $i,j$.

\begin{lem}[Crossing identities and inequalities]
\label{lem:rectcrid}
The following hold.
\begin{enumerate}[a.]
\item For any matrices $M_1,M_2,M_3$ of the appropriate dimensions
\begin{align*}
	&\sum_{i,j,k,l}
	E_{ij} M_1 E_{kl} M_2 
	E_{ij} M_3 E_{kl} =  
	\sum_{i,j,k,l} b_{ij}^2 b_{kl}^2 \,
	(M_1)_{jk} (M_2)_{li} (M_3)_{jk} \, e_i e_l^*,
\\
	&\sum_{i,j,k,l}
	E_{ij} M_1 E_{kl}^* M_2 
	E_{ij} M_3 E_{kl}^* =  
	\sum_{i,j,k,l} b_{ij}^2 b_{kl}^2 \,
	(M_1)_{jl} (M_2)_{ki} (M_3)_{jl} \, e_i e_k^*.
\end{align*}
\item For any $m\times n$ matrix $M$
$$
	\sum_{i,j} \tr[M E_{ij}] E_{ij} =
	\sum_{i,j} b_{ij}^2\, M_{ji}\, e_i e_j^*.
$$
\item Let $\varepsilon_1,\ldots,\varepsilon_4\in\{1,*\}$. Then
for any nonnegative matrices
$M_1,M_2,M_3\geent 0$ of the appropriate dimensions, we have
$$
	\sum_{i,j,k,l}
	E_{ij}^{\varepsilon_1} M_1 E_{kl}^{\varepsilon_2} M_2 
	E_{ij}^{\varepsilon_3} M_3 E_{kl}^{\varepsilon_4} 
	\leent
	\sigma_*^4
	\begin{cases}
	M_2^*M_3^*M_1
	& \text{ if }\varepsilon_1=\varepsilon_3,\varepsilon_2\ne\varepsilon_4,
	\\
	M_3M_1^*M_2^*
	& \text{ if }\varepsilon_1\ne\varepsilon_3,\varepsilon_2=\varepsilon_4,
	\\
	M_3M_2M_1
	& \text{ if }\varepsilon_1\ne\varepsilon_3,\varepsilon_2\ne\varepsilon_4,
	\end{cases}
$$
with equality if $b_{ij}=1$ for all $i,j$.
\end{enumerate}
\end{lem}

\begin{proof}
All parts follow immediately from the definition of $E_{ij}$.
\end{proof}

We now introduce the relevant crossing types in this setting.

\begin{defn}
Let $\pi\in\mathrm{P}_2([2p])$, and fix a crossing $\{i,k\},\{j,l\}\in 
\pi$ such that $i<j<k<l$. Then this crossing is said to be of
\begin{enumerate}[$\bullet$]
\itemsep\abovedisplayskip
\item \emph{type 1} if $i,k$ have opposite parity or $j,l$ 
have opposite parity.
\item \emph{type 2} if $i,k$ have the same parity, $j,l$ have the same 
parity, and there exists 
a pair $\{a,b\}\in\pi$ with $a\in(i,j)\cup(k,l)$ and 
$b\not\in(i,j)\cup(k,l)$;
\item \emph{type 3} otherwise.
\end{enumerate}
\end{defn}

We now consider each crossing type in turn.

\begin{lem}[Type 1 crossing]
\label{lem:rect1cr}
Let $p\in\mathbb{N}$, $\pi\in\mathrm{P}_2([2p])$, and 
$\{a,c\},\{b,d\}\in\pi$ with $a<b<c<d$ be a crossing of type 1. Then 
there exists a pairing $\pi'\in\mathrm{P}_2([2p-4])$, whose definition
depends only on $\pi,a,b,c,d$, so that
$$
	D(\pi)\le \sigma_*^4 D(\pi').
$$
Moreover, equality holds when $b_{ij}=1$ for all $i,j$.
\end{lem}

\begin{proof}
For simplicity, consider the case that $a,b,c$ are odd and $d$ is even; 
the proof in the remaining cases is identical. Then Lemma 
\ref{lem:rectcrid}(c) allows us to estimate
\begin{align*}
	D(\pi) &=
	\max_r 
	\sum_{(\boldsymbol{i},\boldsymbol{j})\sim
	\pi\backslash\{\{a,c\},\{b,d\}\}}
	\sum_{i,j,k,l}
	(M_0 E_{ij} M_1 E_{kl} M_2 E_{ij} M_3 E_{kl}^* M_4)_{rr}
\\	&\le
	\sigma_*^4\,
	\max_r 
	\sum_{(\boldsymbol{i},\boldsymbol{j})\sim
	\pi\backslash\{\{a,c\},\{b,d\}\}}
	(M_0M_2^*M_3^*M_1 M_4)_{rr}
\end{align*}
with
$M_0 := E_{i_1j_1}E_{i_2j_2}^*\cdots E_{i_{a-1}j_{a-1}}^*$,
$M_1 := E_{i_{a+1}j_{a+1}}^*E_{i_{a+2}j_{a+2}}\cdots 
E_{i_{b-1}j_{b-1}}^*$, etc. We readily read off the existence of
$\pi'\in\mathrm{P}_2([2p-4])$ so that the right-hand side equals
$\sigma_*^4D(\pi')$. Moreover, Lemma
\ref{lem:rectcrid}(c) ensures equality when $b_{ij}=1$ for all $i,j$.
\end{proof}

\begin{lem}[Type 2 crossing]
\label{lem:rect2cr}
Let $p\in\mathbb{N}$, $\pi\in\mathrm{P}_2([2p])$, 
$\{a,c\},\{b,d\},\{e,f\}\in\pi$ where $a<b<c<d$, 
$a,c$ have the same parity,
$b,d$ have the same parity, and
$e\in (a,b)\cup (c,d)$,
$f\not\in (a,b)\cup(c,d)$.
Then 
there exists a pairing $\pi'\in\mathrm{P}_2([2p-6])$, whose definition
depends only on $\pi,a,b,c,d,e,f$, so that
$$
	D(\pi)\le \sigma_*^6 D(\pi').
$$
Moreover, equality holds when $b_{ij}=1$ for all $i,j$.
\end{lem}

\begin{proof}
Let us first assume that $a,b,c,d$ are all odd. Then the first 
equation display of Lemma \ref{lem:rectcrid}(a) enables us to estimate
$$
	D(\pi)\le \sigma_*^4 \,
	\max_r \sum_{(\boldsymbol{i},\boldsymbol{j})\sim
        \pi\backslash\{\{a,c\},\{b,d\}\}}
	\tr[M_1^*M_3]
	(M_0M_2^*M_4)_{rr}
$$
with equality if $b_{ij}=1$ for all $i,j$, where $M_k$
are as in the proof of Lemma \ref{lem:rect1cr}. Now note that, by 
assumption, $\{e,f\}$ pairs a term inside the trace with a term inside the 
matrix element on the right-hand side. We can therefore use Lemma 
\ref{lem:rectcrid}(b) or its adjoint (depending on the parities of $e$ and 
$f$) to estimate the right-hand side by $\sigma_*^6D(\pi')$ for some
$\pi'\in\mathrm{P}_2([2p-6])$ that depends only on $\pi,a,b,c,d,e,f$, with
equality if $b_{ij}=1$ for all $i,j$. We omit the details
which are completely analogous to the corresponding argument in the proof 
of Lemma \ref{lem:symmi}.

The other possible parities of $a,b,c,d$ are treated in a completely 
analogous 
way: if $a,c$ are odd and $b,d$ are even we use the second equation
display of Lemma \ref{lem:rectcrid}(a) instead of the first, while the 
remaining two cases use the adjoint of the first or second equation 
display of Lemma \ref{lem:rectcrid}(a), respectively.
\end{proof}

\begin{lem}[Type 3 crossing]
\label{lem:rect3cr}
Let $p\in\mathbb{N}$, $\pi\in\mathrm{P}_2([2p])$, and let
$\{a,c\},\{b,d\}\in\pi$ with $a<b<c<d$ be a crossing of type 3. Then
there exist $\pi_1\in\mathrm{P}_2([b-a+d-c-2])$ and
$\pi_2\in\mathrm{P}_2([2p-2-b+a-d+c])$, which depend
only on $\pi,a,b,c,d$, so that
$$
	D(\pi)\le 
	\begin{cases}
	\sigma_1^2\sigma_*^2 D(\pi_1)D(\pi_2)
	&\text{ if }d\text{ is odd},\\
	\sigma_2^2\sigma_*^2 D(\pi_1)D(\pi_2)
	&\text{ if }d\text{ is even}.
	\end{cases}
$$
Moreover, equality holds when $b_{ij}=1$ for all $i,j$.
\end{lem}

\begin{proof}
Let us first assume that $a,b,c,d$ are all odd. Then the first 
equation display of Lemma \ref{lem:rectcrid}(a) enables us to estimate
\begin{align*}
	D(\pi) &=
	\max_r 
	\sum_{(\boldsymbol{i},\boldsymbol{j})\sim
	\pi\backslash\{\{a,c\},\{b,d\}\}}
	\sum_{i,j,k,l}
	b_{ij}^2 b_{kl}^2
	(M_0)_{ri}(M_1)_{jk}(M_2)_{li}(M_3)_{jk}(M_4)_{lr}
\\
	&\le
	\sigma_*^2\,
	\max_r 
	\sum_l
	\Bigg(
	\sum_{(\boldsymbol{i}_J,\boldsymbol{j}_J)\sim \pi_J}
	(M_0M_2^*)_{rl}(M_4)_{lr}
	\Bigg)
	\Bigg(
	\sum_k
	b_{kl}^2 
	\sum_{(\boldsymbol{i}_I,\boldsymbol{j}_I)\sim \pi_I}
	(M_1^*M_3)_{kk}
	\Bigg)
\end{align*}
with equality when $b_{ij}=1$ for all $i,j$, 
where $M_k$ are as in the proof of Lemma \ref{lem:rect1cr} and
$\pi\backslash\{\{a,c\},\{b,d\}\}=\pi_I\cup \pi_J$ as in the proof of
Lemma \ref{lem:symmii}. The conclusion now follows readily by the same 
argument as in the proof
of Lemma \ref{lem:symmii}.

The other possible parities of $a,b,c,d$ are treated in a completely
analogous way, where we must take care to observe that we gain a factor
$\sigma_1^2$ or $\sigma_2^2$ depending on whether $b,d$ are odd or even, 
respectively.
\end{proof}

We can now complete the proof of Theorem \ref{thm:rectmom}.

\begin{proof}[Proof of Theorem \ref{thm:rectmom}]
We first estimate $\mathbf{E}[\ntr (XX^*)^p]$ as in Lemma \ref{lem:dpi}. 
We then repeatedly apply the following to each quantity $D(\pi)$ in
the resulting expression:
\begin{enumerate}[$\bullet$]
\itemsep\medskipamount
\item If $\pi$ contains a crossing, we apply either Lemma 
\ref{lem:rect1cr}, \ref{lem:rect2cr}, or \ref{lem:rect3cr}
to the smallest crossing in the lexicographic 
order, depending on whether that crossing is of type 1, 2, or 3,
respectively. In the type 2 case, we choose the smallest 
pair $\{e,f\}$ in the lexicographic order that satisfies the assumption of 
Lemma \ref{lem:rect2cr}.
\item If $\pi$ is noncrossing, we apply Lemma \ref{lem:rectnc}.
\end{enumerate}
This algorithm gives rise to an inequality of the form
$$
	\mathbf{E}[\ntr (XX^*)^{p}] \le
	\sum_{k,l\ge 0:k+l\le p}
	\tilde\varkappa_p(k,l)
	\sigma_1^{2k}\sigma_2^{2l}\sigma_*^{2p-2k-2l},
$$
where the coefficients $\tilde\varkappa_p(k,l)\in\mathbb{Z}_+$ are 
independent of the matrix $X$.

Moreover, as all the inequalities used above become equalities when 
$b_{ij}=1$ for all $i,j$, we can apply the same algorithm to compute
$$
	\mathbf{E}[\ntr (YY^*)^{p}] =
	\sum_{k,l\ge 0:k+l\le p}
	\tilde\varkappa_p(k,l)
	d_1^kd_2^l.
$$
The conclusion readily follows from the assumptions 
$\sigma_*=1$, $\sigma_1^2\le d_1$, $\sigma_2^2\le d_2$.
\end{proof}

\subsection{Small deviations}
\label{sec:pfrectsm}

The difficulty as compared with the self-adjoint models discussed in the 
previous sections is to obtain the following moment estimate.

\begin{lem}
\label{lem:rectled}
Define $Y$ as in Theorem \ref{thm:rectmom}, and assume that $d_1\le d_2$.
Then
$$
	\mathbf{E}[\ntr (YY^*)^{p}] \lesssim
	\frac{1}{d_1}
	\big(d_1^{1/2}+d_2^{1/2}\big)^{2p}
	\bigg(
	1+\frac{8p^2}{d_1^{1/2}d_2^{3/2}}
	\bigg)^p
$$
for all $p\ge d_1^{1/6} d_2^{1/2}$.
\end{lem}

\begin{proof}
The result follows immediately from Theorem \ref{thm:rwish} in
section \ref{sec:wishart} below.
\end{proof}

We can now 
essentially repeat the proof of Proposition \ref{prop:ledsymm}.

\begin{prop}
\label{prop:ledrect}
For $X$ as in Model \ref{mod:rectintro} with $\sigma_1\le\sigma_2$, we 
have
$$
        \mathbf{P}\big[
        \|X\| > \big({\textstyle
	\sqrt{\sigma_1^2+\sigma_*^2}+
	\sqrt{\sigma_2^2+\sigma_*^2}}\big)(1+\varepsilon)
        \big]
        \le
        \frac{n\sigma_*^2}{C\sigma_1^2} 
	e^{-\frac{1}{8}\frac{\sigma_1^{1/2}\sigma_2^{3/2}}{\sigma_*^2}
        \varepsilon^{3/2}}
$$
for all $0\le\varepsilon\le 1$, where $C$ is a universal constant.
\end{prop}

\begin{proof}
Suppose first that $\sigma_*=1$, and let $d_1=\lceil \sigma_1^2\rceil$
and $d_2=\lceil \sigma_2^2\rceil$. Using Markov's inequality,
$\|X\|^{2p}\le n\ntr (XX^*)^p$, Theorem \ref{thm:rectmom}, and Lemma 
\ref{lem:rectled}, we obtain
$$
	\mathbf{P}\big[\|X\| \ge
	(d_1^{1/2}+d_2^{1/2})(1+\varepsilon)\big] \lesssim
	\frac{n}{d_1}  e^{-\varepsilon p+\frac{8p^3}{d_1^{1/2}d_2^{3/2}}}
$$
for all $p\ge d_1^{1/6} d_2^{1/2}$, where we used
$(1+\varepsilon)^{-2p}\le 
e^{-\varepsilon p}$ for $0<\varepsilon\le 1$ and $1+x\le e^x$.
If $\varepsilon\ge 64 d_1^{-1/6} d_2^{-1/2}$,
we may choose
$p = \lfloor \frac{1}{4}d_1^{1/4}d_2^{3/4}\sqrt{\varepsilon}\rfloor
\ge d_1^{1/6} d_2^{1/2}$ to obtain
$$
	\mathbf{P}\big[\|X\| \ge
	(d_1^{1/2}+d_2^{1/2})(1+\varepsilon)\big] \lesssim
	\frac{en}{d_1}  e^{-\frac{1}{8}d_1^{1/4}d_2^{3/4}\varepsilon^{3/2}}.
$$
If $\varepsilon< 64 d_1^{-1/6} d_2^{-1/2}$, the same bound is valid as
$$
	\mathbf{P}\big[\|X\| \ge
	(d_1^{1/2}+d_2^{1/2})(1+\varepsilon)\big] \le
	1\lesssim
	\frac{n}{d_1}  e^{-\frac{1}{8}d_1^{1/4}d_2^{3/4}\varepsilon^{3/2}}
$$
using that $\sigma_1^2 \le n\sigma_*^2$ implies $\frac{n}{d_1}\ge 1$.

Combining the above bounds readily yields
$$
	\mathbf{P}\big[\|X\| \ge
	\big({\textstyle \sqrt{\sigma_1^2+1} + \sqrt{\sigma_2^2+1}}\big)
	(1+\varepsilon)\big] \lesssim
	\frac{n}{\sigma_1^2}  
	e^{-\frac{1}{8}\sigma_1^{1/2}\sigma_2^{3/2}\varepsilon^{3/2}}
$$
for $0\le\varepsilon\le 1$. This concludes the proof when $\sigma_*=1$.
For general $\sigma_*$, it suffices to apply the above bound to the
random matrix $\frac{X}{\sigma_*}$.
\end{proof}

We now complete the proof of Theorem \ref{thm:smrect}.

\begin{proof}[Proof of Theorem \ref{thm:smrect}]
Proposition \ref{prop:ledrect} yields
$$
        \mathbf{P}\big[
        \|X\| > 
	\sigma_1+\sigma_2+\tfrac{\sigma_*^2}{\sigma_1}+
	\tfrac{3}{4}\sigma_*^{4/3} \sigma_1^{-1/3}t
        \big]
        \le
        \frac{n\sigma_*^2}{C\sigma_1^2} 
	e^{-\frac{1}{64}t^{3/2}}
$$
by setting $\varepsilon = \frac{1}{4} t 
\sigma_*^{4/3} \sigma_1^{-1/3}\sigma_2^{-1}$,
where we used that $\sigma_1\le\sigma_2$ so that
$$
	{\textstyle
	\sqrt{\sigma_1^2+\sigma_*^2}+\sqrt{\sigma_2^2+\sigma_*^2}} \le
	\sigma_1+\sigma_2+\frac{\sigma_*^2}{\sigma_1} \le
	3\sigma_2.
$$
The conclusion is immediate for
$t\ge \tfrac{4\sigma_*^{2/3}}{\sigma_1^{2/3}}$, and
follows for $t<\tfrac{4\sigma_*^{2/3}}{\sigma_1^{2/3}}\le 4$ as then
$$
	1\le \frac{n\sigma_*^2}{\sigma_1^2} \le
	\frac{en\sigma_*^2}{\sigma_1^2} e^{-\frac{1}{64}t^{3/2}}.
$$
This concludes the proof.
\end{proof}

\subsection{Large deviations}
\label{sec:pfrectlg}

We require the following analogue of Lemma \ref{lem:goeld}.

\begin{lem}
\label{lem:wishld}
Let $Y$ be as in Theorem \ref{thm:rectmom}.
For all $t\ge 0$
$$
	\mathbf{E}\big[\ntr e^{t (YY^*)^{1/2}}\big] \le
	e^{(\sqrt{d_1}+\sqrt{d_2})t + t^2/2}.
$$
\end{lem}

\begin{proof}
We clearly have $\ntr e^{t(YY^*)^{1/2}}\le e^{t\|Y\|}$. It is shown in 
\cite[Theorem 2.13]{DS01} that $\mathbf{E}\|Y\|\le \sqrt{d_1}+\sqrt{d_2}$.
The Gaussian log-Sobolev inequality implies by \cite[eq.\ 
(5.8)]{Led01} and \cite[\S 2.2]{DS01} that $\|Y\|$ is a 
$1$-subgaussian random variable, that is, that
$$
	\mathbf{E}\big[e^{t\{\|Y\|-\mathbf{E}\|Y\|\}}\big]
	\le e^{t^2/2}
$$
for all $t$. 
Combining these facts yields the conclusion.
\end{proof}

We can now prove Theorem \ref{thm:lgrect}.

\begin{proof}[Proof of Theorem \ref{thm:lgrect}]
Suppose that $\sigma_*=1$, and let 
$d_1=\lceil\sigma_1^2\rceil$,
$d_2=\lceil\sigma_2^2\rceil$.
Then
\begin{align*}
	\frac{1}{2}\,\EE\big[\ntr e^{t(XX^*)^{1/2}}\big] &\le
	\EE[\ntr \cosh(t(XX^*)^{1/2})] \\ &\le
	\EE[\ntr \cosh(t(YY^*)^{1/2})] \le
	\EE\big[\ntr e^{t(YY^*)^{1/2}}\big]
\end{align*}
by Theorem \ref{thm:rectmom},
where we used that the Taylor expansion of the hyperbolic cosine only
has terms of even degree. As $e^{t\|X\|} \le n\ntr e^{t(XX^*)^{1/2}}$, 
we obtain
$$
	\EE[e^{t\|X\|}] \le
	2n\, e^{(\sqrt{d_1}+\sqrt{d_2})t + t^2/2}
$$
using Lemma \ref{lem:wishld}.
By Markov's inequality
$$
	\mathbf{P}\big[\|X\| >
	\sqrt{d_1}+\sqrt{d_2}+ \varepsilon\big] \le
	\frac{\mathbf{E}[e^{t\|X\|}]}{
	e^{(\sqrt{d_1}+\sqrt{d_2} + \varepsilon) t}} \le
	2n\, e^{-\varepsilon t+t^2/2}.
$$
Optimizing over $t\ge 0$ yields
$$
	\mathbf{P}\big[\|X\| > 
	{\textstyle \sqrt{\sigma_1^2+1}+\sqrt{\sigma_2^2+1}}
	+\varepsilon\big] \le
	2n\, e^{-\varepsilon^2/2}
$$
for all $\varepsilon\ge 0$.

For general $\sigma_*$, applying the above bound
to the random matrix $\frac{X}{\sigma_*}$ yields
$$
	\mathbf{P}\big[
	\|X\|> {\textstyle\sqrt{\sigma_1^2+\sigma_*^2} +
	\sqrt{\sigma_2^2+\sigma_*^2}}
	+ \sigma_*\varepsilon\big] 
	\le
	2n\, e^{-\varepsilon^2/2}
$$
for all $\varepsilon\ge 0$.
To conclude the proof, it remains to use that 
$\sqrt{\sigma_1^2+\sigma_*^2}+
\sqrt{\sigma_2^2+\sigma_*^2}
\le
\sigma_1+\sigma_2+\frac{\sigma_*^2}{2\sigma_1}+\frac{\sigma_*^2}{2\sigma_2}
\le
\sigma_1+\sigma_2+\sigma_*$.
\end{proof}

\section{Moment estimates for Wishart matrices}
\label{sec:wishart}

The aim of this section is to prove the moment estimate for Wishart 
matrices that was used in Lemma \ref{lem:rectled} above. Throughout this 
section, we fix
$$
	n\le m,\qquad\qquad
	c:=\frac{m}{n}\ge 1.
$$
We let $Y$ be an $n\times m$ matrix with i.i.d.\ $N_{\mathbb{R}}(0,1)$ 
entries, and $Z$ be an $n\times m$ matrix with i.i.d.\ 
$N_{\mathbb{C}}(0,1)$ entries. Our main results are as follows.

\begin{thm}[Complex Wishart moments]
\label{thm:cwish}
For $p\in\mathbb{N}$, we have
$$
	\EE[\ntr (ZZ^*)^p] \lesssim
	n^p (\sqrt{c}+1)^{2p}
	\bigg(
	1+
	\frac{2p^2}{c^{3/2}n^2}
	\bigg)^p
	\frac{c^{3/4}}{p^{3/2}}.
$$
\end{thm}

\begin{thm}[Real Wishart moments]
\label{thm:rwish}
For $p\in\mathbb{N}$, we have
$$
	\EE[\ntr (YY^*)^p] \lesssim
	n^p (\sqrt{c}+1)^{2p}
	\bigg(
	1+
	\frac{8p^2}{c^{3/2}n^2}
	\bigg)^p
	\bigg(
	\frac{c^{3/4}}{p^{3/2}}
	+\frac{1}{n}\bigg).
$$
\end{thm}

A moment estimate for complex Wishart matrices was previously obtained by 
Ledoux in \cite[p.\ 201]{Led04}. However, the constants in Ledoux' 
estimate depend in an unspecified manner on the aspect ratio $c$, making 
it unsuitable for the purposes of this paper.\footnote{The argument of 
\cite[p.\ 201]{Led04} also contains a further issue, that the recursion in 
eq.\ (29) of that paper does not imply the inequality for $b_p$ claimed 
subsequently.} The difficulty in the proofs of Theorems 
\ref{thm:cwish} and \ref{thm:rwish} is to obtain bounds that have optimal
dependence on $c$, which requires a more delicate understanding of the 
structure of the associated moment recursions.

\begin{rem}
To illustrate the sharpness of Theorems \ref{thm:cwish} and \ref{thm:rwish}, 
let us make two observations. First, the argument of section 
\ref{sec:pfrectsm} shows that these moment estimates yield tail bounds 
as in Theorem \ref{thm:smrect}, which match the exact Tracy-Widom 
asymptotics \eqref{eq:wishart} with the optimal order of the fluctuations 
and tail behavior.

On the other hand, if we let $n\to\infty$ with
$p,c$ fixed, it is classical \cite[p.\ 368]{NS06}
that both $n^{-p}\EE[\ntr (YY^*)^p]$ and
$n^{-p}\EE[\ntr (ZZ^*)^p]$ converge to the $p$-moment $\chi_p^c$
of the Marchenko–Pastur distribution. By using the 
explicit formula for its generating function \cite[p.\ 205]{NS06},
the Darboux method \cite[Theorem 5.11]{Wil06} yields 
$$
	\chi_p^c = (1+o(1))\, (\sqrt{c}+1)^{2p} \,
	\frac{c^{1/4}(\sqrt{c}+1)}{2\sqrt{\pi}\,p^{3/2}}
$$
as $p\to\infty$. The estimates of Theorems \ref{thm:cwish} and 
\ref{thm:rwish} reproduce the exact asymptotics of $\chi_p^c$
precisely up to a universal constant.
\end{rem}

The remainder of this section is organized as follows. In section 
\ref{sec:wrec}, we recall the recursive formulas of \cite{HT03,CMSV16} for 
the moments of complex and real Wishart matrices. Let us emphasize at the 
outset that the recursive formula for real Wishart moments involves 
complex Wishart moments, so that we must consider the latter even if one 
is ultimately only interested in real Wishart matrices. We then prove 
Theorem \ref{thm:cwish} in section \ref{sec:pfcw}, and finally prove 
Theorem \ref{thm:rwish} in section \ref{sec:pfrw}.

\subsection{Moment recursions}
\label{sec:wrec}

In addition to the random matrices $Y$ and $Z$ defined above, we also 
introduce an $(n-1)\times (m-1)$ matrix $Z'$ with i.i.d.\ 
$N_{\mathbb{C}}(0,1)$ entries. Throughout the remainder of this section,
we define
$$
	A_p := \frac{\EE[\tr(ZZ^*)^p]}{n^{p+1}},\qquad
	A_p' := \frac{\EE[\tr(Z'Z^{\prime *})^p]}{n^{p+1}},\qquad
	B_p := \frac{\EE[\tr(YY^*)^p]}{n^{p+1}},
$$
so that $\EE[\ntr (ZZ^*)^p]=n^pA_p$ and $\EE[\ntr (YY^*)^p]=n^pB_p$.

The proofs of the main results of this section are based on recursive 
formulas for $A_p,A_p',B_p$ that we presently recall. The following 
recursive formula for $A_p$ was obtained by Haagerup and Thorbj{\o}rnsen 
\cite[Theorem 8.2]{HT03}.

\begin{thm}[Haagerup-Thorbj{\o}rnsen]
\label{thm:ap}
For $p\ge 1$
$$
	A_{p+1} = 2(c+1)\,\frac{2p+1}{2p+4}\,A_p -
	\bigg(
	(c-1)^2 - \frac{p^2}{n^2}
	\bigg)\,
	\frac{p-1}{p+2}\,A_{p-1},
$$
with the initial conditions $A_0=1$ and $A_1=c$.
\end{thm}

A direct consequence is the following.

\begin{cor}
\label{cor:appr}
For $p\ge 1$
$$
	A_{p+1}' = 2\bigg(c+1-\frac{2}{n}\bigg)\,\frac{2p+1}{2p+4}\,A_p' -
	\bigg(
	(c-1)^2 - \frac{p^2}{n^2}
	\bigg)\,
	\frac{p-1}{p+2}\,A_{p-1}',
$$
with $A_0'=\frac{n-1}{n}$ and
$A_1'=\frac{n-1}{n}(c-\frac{1}{n})$. Moreover,
$A_p'\le A_p$ for all $p$.
\end{cor}

\begin{proof}
The recursion for $A_p'$ follows readily from Theorem \ref{thm:ap}.
The inequality $A_p'\le A_p$ follows from Jensen's inequality by the 
convexity of $Z\mapsto \tr (ZZ^*)^p$ by conditioning on the
$(n-1)\times (m-1)$ principal submatrix of $Z$.
\end{proof}

The reason that we consider the modified moments $A_p'$ is that they 
appear in the following moment recursion for $B_p$ due to Cunden et al.\ 
\cite[Theorem 3.5]{CMSV16} (the values of $B_1,B_2$ stated here are 
readily obtained by a straightforward computation).

\begin{thm}[Cunden et al.]
\label{thm:bp}
For $p\ge 2$
\begin{multline*}
	B_{p+1} = 2\bigg(c+1-\frac{1}{n}\bigg)B_p
	-\bigg(
	(c-1)^2 - \frac{4p(p-1)+1}{n^2}
	\bigg)B_{p-1}
\\
	+
	\frac{3}{p-1}
	\bigg[
	\bigg(c+1-\frac{p+1}{n}\bigg) A_p'
	- A_{p+1}'
	\bigg],
\end{multline*}
with $B_0=1$, $B_1=c$, and $B_2=(c+1+\frac{1}{n})c$.
\end{thm}

Finally, we will define in the sequel $\chi_p := 4^{-p}C_p$, where $C_p$ 
denotes the $p$th Catalan number. We recall the well known Catalan 
recursion
\begin{equation}
\label{eq:catalan}
	\chi_{p+1} = \frac{2p+1}{2p+4}\,\chi_p,\qquad\quad
	\chi_0 = 1
\end{equation}
for $p\ge 0$, as well as the standard estimate
$\chi_p \lesssim p^{-3/2}$ by Stirling's formula.

\subsection{Complex Wishart moments}
\label{sec:pfcw}

The aim of this section is to bound the complex Wishart moments $A_p$.
To this end, it will be convenient to define
$$
	K_p := \frac{A_{p+1}}{A_p} \frac{\chi_p}{\chi_{p+1}},
$$
so that
$$
	A_p = 4c\,K_1 K_2\cdots K_{p-1}\chi_p.
$$
It follows readily from Theorem \ref{thm:ap} and \eqref{eq:catalan} that
\begin{equation}
\label{eq:kp}
	K_p = 2(c+1) -
	\bigg(1-\frac{3}{4p^2-1}
	\bigg)
	\bigg(
	(c-1)^2 - \frac{p^2}{n^2}
	\bigg)
	\frac{1}{K_{p-1}}
\end{equation}
for $p\ge 1$. Note that this recursion does not require an initial 
condition, as the second term on the right-hand side vanishes for $p=1$.

The analysis of this equation depends on the sign of the second term on 
the right-hand side. We consider the two cases
$p<(c-1)n$ and $p\ge (c-1)n$ separately.

\subsubsection{The case $p<(c-1)n$}

In this subsection, we fix $1\le p<(c-1)n$ and let
$$
	\lambda := c+1+\sqrt{4c+\frac{p^2}{n^2}},\qquad\qquad
	\bar\lambda := c+1-\sqrt{4c+\frac{p^2}{n^2}}.
$$
Note that $2(c+1)=\lambda+\bar\lambda$ and $(c-1)^2-\frac{p^2}{n^2}=
\lambda\bar\lambda$. In particular, the latter implies that 
$\bar\lambda>0$ as we assumed that $p<(c-1)n$. We can therefore estimate
\begin{equation}
\label{eq:kpineq}
	K_k \le \lambda+\bar\lambda - 
	\bigg(1-\frac{3}{4k^2-1}
	\bigg)
	\frac{\lambda\bar\lambda}{K_{k-1}}
\end{equation}
for all $1\le k\le p$ using \eqref{eq:kp}.

At the core of the proof lie two distinct bounds on $K_k$ for $1\le k\le 
p$. The first bound will be used for large $k$, while the second bound 
will be used for small $k$.

\begin{lem}
\label{lem:clgk}
For $1\le k\le p$, we have
$$
	K_k \le \bigg(1+\frac{2\sqrt{c}}{k^2}\bigg)\lambda.
$$
\end{lem}

\begin{proof}
First, note that
$K_1 = 2(c+1) \le 2\lambda \le (1+2\sqrt{c})\lambda$ as $c\ge 1$.
For $k>1$, we proceed by induction. Assuming we have proved the result
for $k\leftarrow k-1$, we have
$$
	K_k \le 
	\lambda+\bar\lambda - 
	\frac{1-\frac{3}{4k^2-1}}{1+\frac{2\sqrt{c}}{(k-1)^2}}
	\bar\lambda.
$$
using \eqref{eq:kpineq}. To conclude the result, we must therefore show 
that
$$
	\frac{2\sqrt{c}}{(k-1)^2}
	+\frac{3}{4k^2-1}
	\le
	\frac{2\sqrt{c}}{k^2}
	\bigg(1+\frac{2\sqrt{c}}{(k-1)^2}\bigg)
	\frac{\lambda}{\bar\lambda}.	
$$
Subtracting $\frac{2\sqrt{c}}{k^2}$ on both sides and rearranging yields
$$
	2\sqrt{c}\,(2k-1)
	+\frac{3k^2(k-1)^2}{4k^2-1} 
	\le
	2\sqrt{c}\,
	(k-1)^2\,
	\frac{\lambda-\bar\lambda}{\bar\lambda}
	+
	4c\,
	\frac{\lambda}{\bar\lambda}.	
$$
But as
$$
	\frac{\lambda-\bar\lambda}{\bar\lambda} \ge \frac{4}{\sqrt{c}},
	\qquad\quad
	\frac{\lambda}{\bar\lambda}\ge 1,\qquad\quad
	4k^2-1\ge 4k(k-1),
$$
it suffices to prove the quadratic inequality for $k-1$
$$
	7(k-1)^2
	-\bigg(4\sqrt{c}+\frac{3}{4}\bigg) (k-1) 
	+4c-2\sqrt{c} \ge 0.
$$
Thus it suffices to check that the quadratic function has nonpositive 
discriminant, which is readily verified using that $c\ge 1$.
\end{proof}

\begin{lem}
\label{lem:csmk}
For $1\le k\le p$, we have
$$
	K_k \le \bigg(1+\frac{3}{2k}\bigg)\lambda.
$$
\end{lem}

\begin{proof}
Clearly $K_1\le 2\lambda \le (1+\frac{3}{2})\lambda$. For $k>1$ we proceed 
again by induction. Assuming we have proved the result for
$k\leftarrow k-1$, we have
$$
	K_k \le 
	\lambda+\bar\lambda - 
	\frac{1-\frac{3}{4k^2-1}}{1+\frac{3}{2(k-1)}}
	\bar\lambda
$$
using \eqref{eq:kpineq}. We must therefore show that
$$
	\frac{3}{2(k-1)} +
	\frac{3}{4k^2-1}
	\le
	\frac{3}{2k}
	\bigg(1+\frac{3}{2(k-1)}\bigg)
	\frac{\lambda}{\bar\lambda}.
$$
Using that $\lambda\ge\bar\lambda$ and rearranging, it suffices to show 
that
$$
	\frac{1}{2k(k-1)} +
	\frac{1}{4k^2-1}
	\le
	\frac{3}{4k(k-1)}.
$$
This is always true as $4k^2-1\ge 4k(k-1)$.
\end{proof}

Combining the above bounds yields the following conclusion.

\begin{prop}
\label{prop:apsmall}
For $1\le p<(c-1)n$, we have
$$
	A_p \lesssim 
	(\sqrt{c}+1)^{2p} \bigg(1+\frac{p^2}{4c^{3/2}n^2}\bigg)^p
	\frac{c^{3/4}}{p^{3/2}}.
$$
\end{prop}

\begin{proof}
Note that Lemma \ref{lem:csmk} implies
$$
	K_1\cdots K_{\lfloor\sqrt{c}\rfloor} \le
	\lambda^{\lfloor\sqrt{c}\rfloor}
	\prod_{k=1}^{\lfloor\sqrt{c}\rfloor}
	\bigg(1+\frac{3}{2k}\bigg)
	\lesssim
	c^{3/4}\lambda^{\lfloor\sqrt{c}\rfloor},
$$
while Lemma \ref{lem:clgk} implies
$$
	K_{\lfloor\sqrt{c}\rfloor+1}\cdots K_{p-1} \le
	\lambda^{p-1-\lfloor\sqrt{c}\rfloor}
	\prod_{k=\lfloor\sqrt{c}\rfloor+1}^\infty
	\bigg(1+\frac{2\sqrt{c}}{k^2}\bigg)
	\lesssim 
	\lambda^{p-1-\lfloor\sqrt{c}\rfloor}.
$$
We therefore have
$$
	A_p = 4c\,K_1\cdots K_{p-1}\chi_p \lesssim
	c^{3/4}\lambda^p\chi_p,
$$
where we used $\lambda\ge c$. It remains to note that
$$
	\lambda-(\sqrt{c}+1)^2 =
	\sqrt{4c+\frac{p^2}{n^2}} -\sqrt{4c} \le
	\frac{1}{4\sqrt{c}}\frac{p^2}{n^2} \le
	\frac{p^2}{4c^{3/2}n^2}(\sqrt{c}+1)^2
$$
and that $\chi_p \lesssim p^{-3/2}$.
\end{proof}

\begin{rem}
It is instructive to note the features of the analysis 
that were needed to obtain a sharp bound. In Lemma \ref{lem:clgk}, the 
constant $2$ is unimportant but the correct dependence on $c$ is key. In 
contrast, the optimal constant $\frac{3}{2}$ in Lemma \ref{lem:csmk} 
is used in a crucial way to obtain the correct exponent 
$c^{3/4}$ in the final bound.
\end{rem}

\subsubsection{The case $p\ge(c-1)n$}

This case is much easier and yields a qualitatively better bound (the 
latter is however irrelevant for our purposes).

\begin{prop}
\label{prop:aplarge}
For $p\ge (c-1)n$, we have
$$
	A_p \lesssim
	(\sqrt{c}+1)^{2p}\bigg(1+\frac{2p^2}{c^2n^2}\bigg)^p
	\frac{1}{p^{3/2}}.
$$
\end{prop}

\begin{proof}
The assumption $p\ge (c-1)n$ implies
$$
	2(c+1) =
	(\sqrt{c}+1)^2 + \frac{(c-1)^2}{(\sqrt{c}+1)^2}
	\le
	(\sqrt{c}+1)^2 + \frac{p^2}{cn^2}.
$$
Define $N_k := \frac{A_k}{\chi_k}$. Then we can crudely estimate
for $1\le k\le p$
$$
	N_{k+1} \le
	(\sqrt{c}+1)^2 N_k +
	\frac{p^2}{cn^2}\,N_k
	+
	1_{k>1}\,
	\frac{p^2}{n^2}\,
	N_{k-1}
$$
using Theorem \ref{thm:ap} and \eqref{eq:catalan}. To conclude the proof,
it suffices to show this implies
$$
	N_k \le
	4(\sqrt{c}+1)^{2k}\bigg(1+\frac{2p^2}{c^2n^2}\bigg)^k
$$
for all $1\le k\le p$. The claim is trivial for $k=1$ as $N_1=4c$,
while the claim is readily verified to hold for $k>1$ by induction.
\end{proof}

The proof of Theorem \ref{thm:cwish} is now immediate.

\begin{proof}[Proof of Theorem \ref{thm:cwish}]
Combining Propositions \ref{prop:apsmall} and \ref{prop:aplarge} yields
$$
	A_p \lesssim 
	(\sqrt{c}+1)^{2p} \bigg(1+\frac{2p^2}{c^{3/2}n^2}\bigg)^p
	\frac{c^{3/4}}{p^{3/2}}	
$$
using $c\ge 1$. The conclusion follows from the definition of $A_p$.
\end{proof}

\subsection{Real Wishart moments}
\label{sec:pfrw}

Taking inspiration from \cite{Led09}, we define
$$
	D_p := B_p - A_p'.
$$
Then we have the following.

\begin{lem}
\label{lem:dp}
For all $p\ge 1$, we have $D_p\ge 0$ and
\begin{multline*}
	D_{p+1} = 
	2\bigg(c+1-\frac{1}{n}\bigg)D_p -
	\bigg(
	(c-1)^2 - \frac{4p(p-1)+1}{n^2}
	\bigg)D_{p-1} \\
	-\frac{1}{n}A_p'
	+\frac{(3p-1)(p-1)}{n^2}A_{p-1}',
\end{multline*}
with the initial conditions $D_0=\frac{1}{n}$ and 
$D_1=\frac{1}{n}(c+1-\frac{1}{n})$.
\end{lem}

\begin{proof}
To show $D_p\ge 0$, it suffices by Corollary \ref{cor:appr} to show 
that $D_p\ge B_p-A_p\ge 0$, that is, that $\mathbf{E}[\tr (ZZ^*)^p]\le
\mathbf{E}[\tr (YY^*)^p]$. This follows from the Wick 
formula, as $\mathbf{E}[\tr (YY^*)^p]$ may be expressed as a sum over all 
pairings as in Lemma~\ref{lem:rectwick} while in the corresponding 
expression for $\mathbf{E}[\tr (ZZ^*)^p]$ the sum is taken only over those 
pairings that pair even with odd indices (cf.\ Lemma \ref{lem:cplxwick}).

The recursion follows for $p\ge 2$ by applying Corollary 
\ref{cor:appr} to $A_{p+1}'$ on the right-hand side of the recursion
of Theorem \ref{thm:bp}, and a tedious but straightforward simplification
of the resulting expression. That the same expression remains valid for 
$p=1$ can be verified directly using the explicit values for 
$B_0,B_1,B_2,A_0',A_1',A_2'$ that are given in Theorem \ref{thm:bp} and
Corollary \ref{cor:appr}, respectively.
\end{proof}

We will always use Lemma \ref{lem:dp} 
for $1\le k\le p$ in the simplified form
\begin{equation}
\label{eq:dp}
	D_{k+1} \le
	2(c+1)D_k -
	\bigg(
	(c-1)^2 - \frac{4p^2}{n^2}
	\bigg)D_{k-1} +\frac{5(k-1)^2}{n^2}A_{k-1},
\end{equation}
where we used that $4k(k-1)+1\le 4p^2$, $(3k-1)(k-1)\le 5(k-1)^2$,
and $A_{k-1}'\le A_{k-1}$.
As in our analysis of the complex Wishart moments, the following analysis 
will depend on the sign of the second term on the right-hand side.

\subsubsection{The case $2p<(c-1)n$}

In this subsection, we fix $1\le p< \frac{1}{2}(c-1)n$ and let
$$
	\mu := c+1+2\sqrt{c+\frac{p^2}{n^2}},\qquad\qquad
	\bar\mu := c+1-2\sqrt{c+\frac{p^2}{n^2}}.
$$
Note that $2(c+1)=\mu+\bar\mu$ and $(c-1)^2-\frac{4p^2}{n^2}=\mu\bar\mu$.
Thus $\bar\mu>0$ as $2p<(c-1)n$.

\begin{lem}
\label{lem:dppert}
We have
$$
	D_p \le \frac{\mu^p}{n} +
	\sum_{l=1}^{p-1}
	\frac{\mu^{p-l}-\bar\mu^{p-l}}{\mu-\bar\mu}\,
	\frac{5(l-1)^2}{n^2}A_{l-1}.
$$
\end{lem}

\begin{proof}
We can equivalently write \eqref{eq:dp} as
$$
	D_{k+1}-\mu D_k \le
	\bar\mu (D_k-\mu D_{k-1}) 
	+\frac{5(k-1)^2}{n^2}A_{k-1}
$$
for $1\le k\le p$. Iterating this inequality yields
$$
	D_{k+1}-\mu D_k \le
	\sum_{l=1}^k \bar\mu^{k-l}\, \frac{5(l-1)^2}{n^2}A_{l-1}
$$
for $1\le k\le p$, where we used that $\bar\mu>0$ and
$D_1-\mu D_0<0$. Iterating again yields
\begin{align*}
	D_{k+1} &\le
	\mu^k D_1 + 
	\sum_{r=1}^k \mu^{k-r}\sum_{l=1}^r 
	\bar\mu^{r-l}\,\frac{5(l-1)^2}{n^2}A_{l-1}
	\\
	&=
	\mu^k D_1 +
	\sum_{l=1}^k \frac{\mu^{k-l+1}-\bar\mu^{k-l+1}}{\mu-\bar\mu}
	\,\frac{5(l-1)^2}{n^2}A_{l-1}
\end{align*}
for $1\le k\le p$. The conclusion follows as
$\frac{D_1}{\mu}\le \frac{1}{n}$.
\end{proof}

We can now use Proposition \ref{prop:apsmall} to estimate $D_p$.

\begin{prop}
\label{prop:dpsmall}
For $1\le p< \frac{1}{2}(c-1)n$, we have
$$
	D_p \lesssim
	\frac{1}{n} (\sqrt{c}+1)^{2p}\bigg(
	1+\frac{p^2}{c^{3/2}n^2}
	\bigg)^p.
$$
\end{prop}

\begin{proof}
We can estimate as in the proof of Proposition \ref{prop:apsmall}
$$
	\mu \le (\sqrt{c}+1)^2\bigg(
	1+\frac{p^2}{c^{3/2}n^2}
	\bigg)
$$
Thus the only difficulty is to bound the second term on the right-hand 
side of the inequality of Lemma \ref{lem:dppert}. To this end, we estimate
\begin{multline*}
	\sum_{l=1}^{p-1}
	\frac{\mu^{p-l}-\bar\mu^{p-l}}{\mu-\bar\mu}\,
	\frac{5(l-1)^2}{n^2}A_{l-1}
\\
	\lesssim
	(\sqrt{c}+1)^{2(p-1)}
	\bigg(
        1+\frac{p^2}{c^{3/2}n^2}
        \bigg)^{p-1}
	c^{1/4}
	\sum_{r=1}^{p-2}
	\bigg(\frac{
		1+\frac{p^2}{4c^{3/2}n^2}
	}{
		1+\frac{p^2}{c^{3/2}n^2}
	}
	\bigg)^r
	\frac{r^{1/2}}{n^2}
\end{multline*}
using $\mu-\bar\mu \ge 4\sqrt{c}$, the above inequality for $\mu$, and
Proposition \ref{prop:apsmall}.
We now split the sum into parts with $r\le n^{2/3}\sqrt{c}$ and
$r>n^{2/3}\sqrt{c}$. For the first part, we have
$$
	\sum_{r=1}^{\lfloor n^{2/3}\sqrt{c}\rfloor}
	\bigg(\frac{
		1+\frac{p^2}{4c^{3/2}n^2}
	}{
		1+\frac{p^2}{c^{3/2}n^2}
	}
	\bigg)^r
	\frac{r^{1/2}}{n^2}
	\le
	\frac{1}{n^2}
	\sum_{r=1}^{\lfloor n^{2/3}\sqrt{c}\rfloor}
	r^{1/2}
	\lesssim
	\frac{c^{3/4}}{n}.
$$
For the second part, we have
\begin{align*}
	\sum_{r=\lfloor n^{2/3}\sqrt{c}\rfloor+1}^{p-2}
	\bigg(\frac{
		1+\frac{p^2}{4c^{3/2}n^2}
	}{
		1+\frac{p^2}{c^{3/2}n^2}
	}
	\bigg)^r
	\frac{r^{1/2}}{n^2}
	&\le
	\frac{c^{3/4}}{n}
	\sum_{r=\lfloor n^{2/3}\sqrt{c}\rfloor+1}^{p-2}
	\bigg(\frac{
		1+\frac{p^2}{4c^{3/2}n^2}
	}{
		1+\frac{p^2}{c^{3/2}n^2}
	}
	\bigg)^r
	\frac{r^2}{c^{3/2}n^2}  
\\
	&\le
	\frac{c^{3/4}}{n}
	\frac{p^2}{c^{3/2}n^2}
	\sum_{r=0}^\infty
	\bigg(
	\frac{
		1+\frac{p^2}{4c^{3/2}n^2}
	}{
		1+\frac{p^2}{c^{3/2}n^2}
	}
	\bigg)^r
\\
	&=
	\frac{4}{3}\frac{c^{3/4}}{n}
	\bigg(1+\frac{p^2}{c^{3/2}n^2}\bigg).
\end{align*}
Combining the above estimates with Lemma \ref{lem:dppert}
readily yields the conclusion.
\end{proof}

\subsubsection{The case $2p\ge(c-1)n$}

We need the following counterpart of Lemma \ref{lem:dppert}.

\begin{lem}
\label{lem:dpbigpert}
For $p\ge 1$ such that $2p\ge (c-1)n$, we have
$$
	D_p \le C^{p-1} D_1 + \sum_{l=1}^{p-1} C^{p-1-l}\,
	\frac{5(l-1)^2}{n^2}A_{l-1}
$$
with
$$
	C := (\sqrt{c}+1)^2\bigg(1+\frac{8p^2}{c^2n^2}\bigg).
$$
\end{lem}

\begin{proof}
The assumption $2p\ge (c-1)n$ implies
$$
	2(c+1) 
	\le
	(\sqrt{c}+1)^2 + \frac{4p^2}{cn^2}
	\le
	(\sqrt{c}+1)^2\bigg(1+ \frac{4p^2}{c^2n^2}\bigg)
$$
as in the proof of Proposition \ref{prop:aplarge}.
Thus \eqref{eq:dp} yields
\begin{equation}
\label{eq:dkplonebig}
	D_{k+1} \le
	(\sqrt{c}+1)^2\bigg(1+ \frac{4p^2}{c^2n^2}\bigg)
	D_k 
	+
	(\sqrt{c}+1)^2\frac{4p^2}{cn^2} D_{k-1} +\frac{5(k-1)^2}{n^2}A_{k-1}
\end{equation}
for $1\le k\le p$. We will show by induction that for $1\le k\le p$
$$
	D_k \le 
	C^{k-1}D_1 + \sum_{l=1}^{k-1} C^{k-1-l}\,
	\frac{5(l-1)^2}{n^2}A_{l-1}.
$$
The claim is trivial for $k=1$, and follows for $k=2$ from 
\eqref{eq:dkplonebig} using $cD_0\le D_1$. On the other hand, if the claim 
has been verified up to $k$, then \eqref{eq:dkplonebig} yields
\begin{align*}
	&D_{k+1} \le
	(\sqrt{c}+1)^2
	\bigg[
	\bigg(1+ \frac{4p^2}{c^2n^2}\bigg)C
	+
	\frac{4p^2}{cn^2} 
	\bigg]
	\bigg[
	C^{k-2}D_1
	+
	\sum_{l=1}^{k-2} C^{k-2-l}\,
        \frac{5(l-1)^2}{n^2}A_{l-1}
	\bigg]
	\\
	&\qquad\qquad\qquad
	+
	(\sqrt{c}+1)^2\bigg(1+ \frac{4p^2}{c^2n^2}\bigg)
        \frac{5(k-2)^2}{n^2}A_{k-2}
	+\frac{5(k-1)^2}{n^2}A_{k-1}
\\
	&\le
	C^2
	\bigg[
	C^{k-2}D_1
	+
	\sum_{l=1}^{k-2} C^{k-2-l}\,
        \frac{5(l-1)^2}{n^2}A_{l-1}
	\bigg]
	+
        C\,\frac{5(k-2)^2}{n^2}A_{k-2}
	+\frac{5(k-1)^2}{n^2}A_{k-1}
\\
	&= 
	C^{k}D_1 + \sum_{l=1}^{k} C^{k-l}\,
	\frac{5(l-1)^2}{n^2}A_{l-1},
\end{align*}
concluding the proof of the claim for $k+1$. The case $k=p$ yields the 
result.
\end{proof}

We now proceed as in the proof of Proposition \ref{prop:dpsmall}.

\begin{prop}
\label{prop:dplg}
For $p\ge\frac{1}{2}(c-1)n$, we have
$$
	D_p \lesssim
	\frac{1}{n} (\sqrt{c}+1)^{2p}\bigg(
	1+\frac{8p^2}{c^{3/2}n^2}
	\bigg)^p.
$$
\end{prop}

\begin{proof}
The conclusion is trivial if $p=0$. For $p\ge 1$, we obtain
$$
	D_p \lesssim \frac{C^p}{n} + 
	(\sqrt{c}+1)^{2(p-2)} 
	\bigg(1+\frac{8p^2}{c^{3/2}n^2}\bigg)^{p-2}
	c^{3/4}
	\sum_{r=1}^{p-2} 
	\bigg(\frac{1+\frac{2p^2}{c^{3/2}n^2}}{
	1+\frac{8p^2}{c^{3/2}n^2}}\bigg)^r
	\frac{r^{1/2}}{n^2}
$$
using Lemma \ref{lem:dpbigpert}, $\frac{D_1}{C}\le \frac{1}{n}$,
$C\le (\sqrt{c}+1)^2\big(1+\frac{8p^2}{c^{3/2}n^2}\big)$, and
Theorem \ref{thm:cwish}.
The conclusion now follows exactly as in the proof of
Proposition \ref{prop:dpsmall}.
\end{proof}

The proof of Theorem \ref{thm:rwish} is now immediate.

\begin{proof}[Proof of Theorem \ref{thm:rwish}]
Combining Propositions \ref{prop:dpsmall} and \ref{prop:dplg} yields
$$
	B_p \le
	D_p + A_p \lesssim
	(\sqrt{c}+1)^{2p}\bigg(
	1+\frac{8p^2}{c^{3/2}n^2}
	\bigg)^p
	\bigg(\frac{c^{3/4}}{p^{3/2}}+\frac{1}{n}\bigg),
$$
using $A_p'\le A_p$ and Theorem \ref{thm:cwish}.
The result follows from the definition of $B_p$.
\end{proof}

\subsection*{Acknowledgment}

This work was supported in part by NSF grants DMS-1856221, 
DMS-2054565, and DMS-2347954. We thank the referee for helpful 
suggestions.

\bibliographystyle{abbrv}
\bibliography{ref}

\end{document}